\documentclass[a4paper]{amsart}
\usepackage{amsmath,amssymb,amsfonts}
\usepackage[all]{xy}
\numberwithin{equation}{section}
\theoremstyle{plain}
 \newtheorem{thm}{Theorem}[section]
 \newtheorem{cor}[thm]{Corollary}
 \newtheorem{lem}[thm]{Lemma}
 \newtheorem{prop}[thm]{Proposition}
\theoremstyle{definition}
 
 \newtheorem{exmp}[thm]{Example}
\theoremstyle{remark}
 \newtheorem{rem}[thm]{Remark}
\DeclareMathOperator{\Ad}{Ad}
\DeclareMathOperator{\Ann}{Ann}
\DeclareMathOperator{\trace}{Trace}
\DeclareMathOperator{\sm}{sym}
\DeclareMathOperator{\Ind}{Ind}
\DeclareMathOperator{\ord}{ord}
\def\ang#1#2{\langle #1,#2\rangle}

\def\udots{\mathinner{\mkern1mu\raise1pt\vbox{\kern7pt\hbox{.}}\mkern2mu
    \raise4pt\hbox{.}\mkern2mu\raise7pt\hbox{.}\mkern1mu}}
\def\scirc{\raise2pt\hbox{${}_\circ$}}
\allowdisplaybreaks

\def\pf{{\em Proof}.\ }
\begin{document}

\title[Boundary value problems on Riemannian Symmetric Spaces]
{Boundary value problems on Riemannian Symmetric Spaces of the noncompact Type}
\author{Toshio Oshima}
\address{Graduate School of Mathematical Sciences,
University of Tokyo, 7-3-1, Komaba, Meguro-ku, Tokyo 153-8914, Japan}
\email{oshima@ms.u-tokyo.ac.jp}
\author{Nobukazu Shimeno}
\address{School of Science and Technology, 
Kwansei Gakuin University, 
Gakuen, Sanda, Hyogo 669-1337, Japan}
\email{shimeno@kwansei.ac.jp}
\begin{abstract}
We characterize the image of the Poisson transform on  
each boundary component of a Riemannian symmetric space 
of the noncompact type by a system of differential equations. 
The system corresponds to  a generator system of a two sided 
ideals of an universal enveloping algebra, 
which are explicitly given by analogues of minimal polynomials of matrices. 
\end{abstract}
\maketitle
\section{Introduction}
The classical Poisson integral of a function on the unit circle in  
the complex plane gives a harmonic function on the unit disk. 
More generally, each eigenfunction of the Laplace-Beltrami operator on 
the Poincar\'e disk can be represented by a generalized Poisson integral of 
a hyperfunction on the unit circle. 

The notion of the Poisson integral 
is generalized to a Riemannian symmetric space $X=G/K$ of the noncompact type, 
where $G$ is a connected real reductive Lie group and $K$ its 
maximal compact subgroup. 
The so called ``Helgason conjecture'' states that each joint eigenfunction of the 
invariant differential operators on $X$ has a Poisson integral representation 
by a 
hyperfunction on the Furstenberg boundary $G/P$ of $X$, where $P$ is 
a minimal parabolic subgroup of $G$. 
Helgason proved the conjecture for 
the Poincar\'e disk. 
Kashiwara {\em et al.}~\cite{K--} prove it generally  
by using the theory of hyperfunctions and 
the system of differential equations with regular singularities and 
their boundary value problem due to 
Kashiwara and Oshima \cite{KO}. 

The Poisson transform is an intertwining operator 
from the spherical principal series representation to the 
eigenspace representation. 
For generic parameter $\lambda$ 
of the principal series representation, 
the Poisson transform $\mathcal{P}_\lambda$ 
gives an isomorphism of the representations. 
The principal series representation is 
realized on the space of the sections of a homogeneous line bundle over 
$G/P$, whose parameter is $\lambda$. 
If $\lambda=\rho$, then the line bundle is trivial and the representation is 
realized on the space of functions on $G/P$. Then the 
image of $\mathcal{P}_\rho$ consists of the harmonic 
functions, that is the functions which are annihilated by 
 the invariant differential 
operators on the symmetric space that kill the constant functions. 
We call this the ``harmonic case''. 

It is natural to pose the problem of characterizing the image of 
$\mathcal{P}_\lambda$ when the map is not bijective. 
An interesting case corresponds to the problem of characterizing 
the image of the Poisson transform from another boundary component 
of $X$ in one of Satake compactifications of $X$ 
(cf. \cite{Sa}, \cite{O2}). 
Each boundary component is of the form $G/P_\Xi$, 
where $P_\Xi$ is a parabolic subgroup of $G$. 
The Furstenberg boundary is the maximal among 
the boundary components. 

For a classical Hermitian symmetric space of tube type, 
Hua \cite{Hua} studied the Poisson integrals of functions on the Shilov boundary, 
which are generalization of the classical Poisson integrals on the 
unit disk. The Poisson integrals are harmonic functions, and 
moreover, they are annihilated by second order differential 
operators, 
which are called the Hua operators. 
Kor\'anyi and Malliavin \cite{KM}, and Johnson \cite{J1} 
showed that the Hua operators characterize the Poisson integrals 
of hyperfunctions on the Shilov boundary of the Siegel upper half planes. 
Johnson and Kor\'anyi \cite{JK} constructed the Hua operators for 
a Hermitian symmetric space of tube type in general and 
proved that they characterize the Poisson integrals 
of hyperfunctions on the Shilov boundary. 
The second author~\cite{Sn2} generalized the result to 
non-harmonic cases. In \cite{Sn4},
 he also constructed a system of differential equations 
that characterizes the image of the Poisson transform from 
a certain kind of boundary component of a Hermitian symmetric space. 

For a Hermitian symmetric space of non-tube type, Berline and Vergne~\cite{BV}  
defined generalized Hua operators, which are third order 
differential operators, 
and proved that these operators with 
invariant differential operators characterize the Poisson integrals of  
hyperfunctions on the Shilov boundary, in the harmonic case. 
Koufany and Zhang~\cite{KZ} generalize the result to non-harmonic cases. 
For $G=U(p,q)$ these authors also showed that second order operators 
characterize the image of the Poisson transform from the 
Shilov boundary, even for non-tube cases, that is the case of $p>q$. 

Johnson constructed 
a system of differential equations that characterizes the image 
of the Poisson transform from each boundary component 
for $G=SL(n,\mathbb{R})$ and $SL(n,\mathbb{C})$ in 
\cite{J2}, and for general $G$ in \cite{J3}, in the harmonic case. 

The first author~\cite{O1} proposed a method to study boundary value 
problems for various boundaries of $X$. 
He constructed a system of differential operators 
corresponding to the boundary $GL(n,\mathbb{R})/P_{n-1,1}$ of 
$GL(n,\mathbb{R})/O(n)$, where $P_{n-1,1}$ is the  maximal parabolic 
subgroup of $GL(n,\mathbb{R})$ corresponding to the partition $(n-1,1)$. 
To prove that the differential equations indeed characterize the image 
of the Poisson transform, he used the method of calculating 
differential equations for boundary values on the Fusrtenberg boundary, 
which are called ``induced equations''. All of the above mentioned 
works on the problem of characterizing the image of the Poisson transform 
from a boundary component by a system of differential equations use 
essentially the method of calculating induced equations. 

On the other hand, recently 
the first author (\cite{O4,O5,O6} and \cite{OO} with Oda), 
studied two sided ideals of a universal enveloping 
algebra of a complex reductive Lie algebra, which 
are annihilators of generalized Verma modules, and apply them 
to boundary value problems for various boundaries of a symmetric 
space. 
In this paper, 
we use two sided ideals constructed explicitly 
in \cite{O6, OO} to characterize 
the image of the Poisson transform from a boundary component of a 
symmetric space, giving several examples including previously known cases. 
We also study the case of homogeneous 
line bundle on a Hermitian symmetric space. 
Since the differential operators come from a two sided ideal of the universal 
enveloping algebra of the complexification of the Lie algebra of $G$, 
the proof that they characterize the image of the Poisson transform 
is fairly easy. Indeed we do not need to calculate induced equations 
on the Furstenberg boundary. 
For the harmonic case, our operators are different from those constructed 
by Johnson \cite{J2, J3} and more explicit. 

This paper is organized as follows. 
In \S 2 we review on representations realized on a symmetric 
space and give basic results on the Poisson transforms on 
various boundaries. 

In \S 3 we review on minimal polynomials on complex reductive Lie 
algebras, which give a generator system of the annihilator 
of a generalized Verma module after \cite{O6, OO} and 
show that the corresponding differential operators on 
a Riemannian symmetric space characterize the image of 
the Poisson transform from a boundary component of the 
symmetric space. 

In \S 4, we give examples when $G$ is $U(p,q)$, $Sp(n,\mathbb{R})$ 
or $GL(n,\mathbb{R})$. In particular, 
for $G=U(p,q)$ or $Sp(n,\mathbb{R})$ and $G/P_\Xi$ the 
Shilov boundary of $X$, our operators for the trivial 
line bundle over $X=G/K$ coincide with 
the previously known ``Hua operators'' mentioned above. 

\section{Representations on symmetric spaces}\label{sec:Rep}

In this section we review on representations realized on 
Riemannian symmetric spaces
of the noncompact type and their characterizations by differential equations.
The statements in this section are known results or at least a reformulation 
or an easy consequence of known facts (cf.~\cite{He1}, \cite{He2}, 
\cite{K--}, \cite{Ko}, \cite{KR}, \cite{O1}, \cite{Sn1} etc.) and this 
section can be read without referring to other sections.

Let $G$ be a connected real connected semisimple Lie group, possibly  
with infinite center. 
Let $K$ be a maximal compact subgroup of $G$ modulo the center of $G$, 
$\theta$ be the corresponding Cartan involution and 
$\mathfrak g = \mathfrak k + \mathfrak p$ be the corresponding Cartan decomposition of 
the Lie algebra $\mathfrak g$ of $G$.
Fix a maximal abelian subspace $\mathfrak a_{\mathfrak p}$ of $\mathfrak p$.
Let $\Sigma(\mathfrak a_{\mathfrak p})$ be the set of the roots defined by the pair
$(\mathfrak g,\mathfrak a_{\mathfrak p})$ and fix a positive system 
$\Sigma(\mathfrak a_{\mathfrak p})^+$.
We denote its Weyl group by $W(\mathfrak a_{\mathfrak p})$ and 
the fundamental system by $\Psi(\mathfrak a_{\mathfrak p})$, 
and the half of the sum of the positive roots 
counting their multiplicities is denoted by $\rho$.
Let $G=KAN$ be the Iwasawa decomposition of $G$ 
with $\text{Lie}(A)=\mathfrak{a}_\mathfrak{p}$ and 
$N$ corresponding to 
$\Sigma(\mathfrak a_{\mathfrak p})^+$.
Then $P=MAN$ is a minimal parabolic subgroup of $G$.
Here $M$ is the centralizer of $\mathfrak a_{\mathfrak p}$ in $K$.
We denote by $\mathfrak k$, $\mathfrak m$ and $\mathfrak n$ the Lie algebras of $K$, $M$ 
and $N$, respectively.

Let $U(\mathfrak g)$ be the universal enveloping algebra of the complexification
$\mathfrak g_{\mathbb C}$ of $\mathfrak g$, which we identify with the algebra of left
invariant differential operators on $G$.
In general, for a subalgebra $\mathfrak l$ of $\mathfrak g$, we denote by
$U(\mathfrak l)$ the universal enveloping algebra of the complexification
$\mathfrak l_{\mathbb C}$ of $\mathfrak l$.
Let $S(\mathfrak g)$ be the symmetric algebra of $\mathfrak g_{\mathbb C}$.
Then the map $\sm$ of symmetrization of $S(\mathfrak g)$ to $U(\mathfrak g)$ defines
a $K$-linear bijection.
By the Killing form on $\mathfrak g_{\mathbb C}$ we identify the space 
$\mathcal O(\mathfrak p_{\mathbb C})$ of polynomial functions on the complexification 
$\mathfrak p_{\mathbb C}$ of $\mathfrak p$ with the symmetric algebra of 
$\mathfrak p_{\mathbb C}$.
Let $\mathcal O(\mathfrak p)^K$ be the space of all $K$-invariant polynomials on 
$\mathfrak p_{\mathbb C}$ and $\mathcal H$ be the space of all $K$-harmonic 
polynomials on $\mathfrak p_{\mathbb C}$.
Then we have the following $K$-linear bijection
\begin{equation}\label{eq:4.1}
\begin{aligned}
 \mathcal H\otimes \mathcal O(\mathfrak p)^K\otimes U(\mathfrak k)\ &\widetilde\rightarrow\ 
    U(\mathfrak g)\\
 h\otimes p\otimes k\ &\mapsto\ 
    \sm(h)\otimes \sm(p)\otimes k
\end{aligned}
\end{equation}
because of the Cartan decomposition $\mathfrak g = \mathfrak k + \mathfrak p$ and the 
$K$-linear bijection 
\[
\mathcal H\otimes\mathcal O(\mathfrak p)^K\ni h\otimes p\mapsto hp\in
\mathcal O(\mathfrak p)
\]
studied by \cite{KR}.

We denote by $\mathcal A(G)$ and $\mathcal B(G)$
the space of real analytic functions and that of hyperfunctions on $G$,
respectively.
Then they are left $G$-modules by $(\pi_gf)(x) = f(g^{-1}x)$ for $g\in G$
and the functions $f$ in the spaces.
We write $\mathcal A(G)_K$ for the space of all the $K$-finite elements of 
$\mathcal A(G)$.

By the decomposition
\begin{equation}
\label{eqn:enviwasawa}
 U(\mathfrak g) = \mathfrak nU(\mathfrak n + \mathfrak a_{\mathfrak p}) \oplus U(\mathfrak a_{\mathfrak p})\oplus
              U(\mathfrak g)\mathfrak k
\end{equation}
coming from the Iwasawa decomposition $\mathfrak g=\mathfrak k + \mathfrak a_{\mathfrak p} + \mathfrak n$,
we define $D_{\mathfrak a_{\mathfrak p}}\in U(\mathfrak a_{\mathfrak p})$ for $D\in U(\mathfrak g)$ so that
$D - D_{\mathfrak a_{\mathfrak p}}\in\mathfrak nU(\mathfrak n + \mathfrak a_{\mathfrak p}) + U(\mathfrak g)\mathfrak k$
and put $\gamma(D) = e^{-\rho}\scirc D_{\mathfrak a_{\mathfrak p}}\scirc e^\rho$.
Here $e^{\rho}$ is the function on $A$ defined by $e^{\rho}(a) = a^\rho$.

Note that the kernel of the restriction of $\gamma$ to the space of
all the $K$-invariants $U(\mathfrak g)^K$ of $U(\mathfrak g)$ equals 
$U(\mathfrak g)^K\cap U(\mathfrak g)\mathfrak k$ and the restriction defines the 
Harish-Chandra isomorphism
\begin{equation}\label{def:bargamma}
 \bar\gamma: \mathbb D(G/K)\simeq 
             U(\mathfrak g)^K/(U(\mathfrak g)^K\cap U(\mathfrak g)\mathfrak k)
             \rightarrow U(\mathfrak a_{\mathfrak p})^W
\end{equation}
onto the space $U(\mathfrak a_{\mathfrak p})^W$ of all the $W(\mathfrak a_{\mathfrak p})$-invariants in 
$U(\mathfrak a_{\mathfrak p})$.
Here $\mathbb D(G/K)$ is the algebra of invariant differential operators on
$G/K$.
Note also that $\sm(O(\mathfrak p)^K)$ is isomorphic to $\mathbb D(G/K)$ through
(\ref{eqn:enviwasawa}) and (\ref{def:bargamma}) 
as $K$-modules.

Identifying $U(\mathfrak a_{\mathfrak p})$ with the space of polynomial functions on 
the complex dual $\mathfrak a_{\mathbb C}^*$ of $\mathfrak a_{\mathfrak p}$, we put
\begin{equation}
 \gamma_\lambda(D) = \gamma(D)(\lambda)\in\mathbb C
\end{equation}
for $\lambda\in\mathfrak a_{\mathbb C}^*$ and $D\in U(\mathfrak g)$.
Now we define
\begin{equation}\label{eq:4.5}
 J_\lambda = U(\mathfrak g)\mathfrak k
     + \sum_{p\in\mathcal O(\mathfrak p)^K}U(\mathfrak g)(\sm(p)-\gamma_\lambda(\sm(p)))
\end{equation}
and
\begin{equation}\label{eq:4.6}
 \mathcal A(G/K;\mathcal M_\lambda) =
 \{u\in\mathcal A(G); Du = 0\quad\text{for }D\in J_\lambda\}
\end{equation}
and put 
$\mathcal A(G/K;\mathcal M_\lambda)_K = \mathcal A(G)_K\cap \mathcal A(G/K;\mathcal M_\lambda)$.
Here $\mathcal A(G/K; \mathcal M_\lambda)$ is naturally a subspace of the space 
$\mathcal A(G/K)$ of real analytic functions on $G/K$ because the function in 
$\mathcal A(G/K; \mathcal M_\lambda)$ is right $K$-invariant.
Now we can state our basic theorem.

\begin{thm}\label{thm:4.1}
Fix $\lambda\in\mathfrak a_{\mathbb C}^*$ and define a bilinear form
\begin{equation}
\begin{aligned}
  \mathcal H\otimes\mathcal A(G/K;\mathcal M_\lambda)\ &\rightarrow\ \mathbb C\\
  (h,u)\ &\mapsto\ 
   \ang{h}{u} = (\sm(h)u)(e)
\end{aligned}
\end{equation}
and for a subspace $V$ of $\mathcal A(G/K;\mathcal M_\lambda)$, put
\begin{equation}
 H(V) = \{h\in\mathcal H; \ang{h}{u} = 0\quad\text{for }u\in V\}.
\end{equation}

{\rm i)} The bilinear form $\ang{\ }{\ }$ is $K$-invariant and non-degenerate.

{\rm ii)} If $V$ is a subspace of $\mathcal A(G/K;\mathcal M_\lambda)_K$, then
\[
 V = \{u\in\mathcal A(G/K;\mathcal M_\lambda)_K; \ang{h}{u}=0
      \quad\text{for }h\in H(V)\}.
\]

{\rm iii)} There are natural bijections between the following sets of modules.
\begin{align*}
 \mathfrak V(\lambda) &= \{V\subset\mathcal A(G/K;\mathcal M_\lambda);
  V\text{ is a close subspace of }C^\infty(G)\text{ and }G
   \text{-invariant}\},\\
 \mathfrak V(\lambda)_K &= \{V_K\subset\mathcal A(G/K;\mathcal M_\lambda)_K;
  V_K\text{ is a }\mathfrak g\text{-invariant subspace}\},\\
 \mathfrak J(\lambda) &= \{ J\supset J_\lambda; 
  J\text{ is a left ideal of }U(\mathfrak g)\}.
\end{align*}
Here the bijections are given by
\begin{align}
 \mathfrak V(\lambda)\ni V\ &\mapsto\ 
 V\cap\mathcal A(G)_K\in\mathfrak V(\lambda)_K,\label{eq:4.9}\\
 \mathfrak V(\lambda)_K\ni V_K\ &\mapsto\ 
 J_\lambda + \sum_{p\in H(V_K)}U(\mathfrak g)\sm(p)\in\mathfrak J(\lambda)\label{eq:4.10}\\
 \mathfrak J(\lambda)\ni J\ &\mapsto\ 
 \{u\in\mathcal A(G); Du=0\text{ for }D\in J\}\in\mathfrak V(\lambda)\label{eq:4.11}
\end{align}
\end{thm}

Before the proof of this theorem we review on the Poisson transform.
The $G$-module
\begin{align}\label{eq:4.12}
 \mathcal B(G/P;\mathcal L_\lambda)  =
 \{f\in\mathcal B(G)\,;\, & f(gman)  = a^{\lambda - \rho}f(g) \\ 
 & \,\text{ for }(g,m,a,n)   \in G\times M\times A\times N\} \notag
\end{align}
is the space of hyperfunction sections of spherical principal series of $G$
parametrized by $\lambda\in\mathfrak a_{\mathbb C}^*$.
Put $\mathcal A(G/P;\mathcal L_\lambda) = \mathcal B(G/P;\mathcal L_\lambda)\cap\mathcal A(G)$.
Define the $K$-fixed element $K\times A\times N\ni(k,a,n)\mapsto 
{\mathbf 1}_\lambda(kan) = a^{\lambda-\rho}$ 
of $\mathcal A(G/P;\mathcal L_\lambda)$ and put 
$P_\lambda(g) = {\mathbf 1}_{-\lambda}(g^{-1})$.
By the $G$-invariant bilinear form
\begin{equation}
\begin{aligned}
  \mathcal B(G/P;\mathcal L_\lambda)\times \mathcal A(G/P;\mathcal L_{-\lambda})
  \ &\rightarrow\ \mathbb C\\
  (\phi, f)\ &\mapsto\ \ang{\phi}{f}_\lambda = \int_K\phi(k)f(k)dk
\end{aligned}
\end{equation}
with the normalized Haar measure $dk$ on $K$, we define the Poisson transform
\begin{equation}\label{eq:4.14}
\begin{aligned}
 \mathcal P_\lambda: \mathcal B(G/P;\mathcal L_\lambda)\ &\rightarrow\ \mathcal B(G)\\
 \phi\ &\mapsto\ \mathcal P_\lambda\phi(g)
 = \ang{\pi_{g^{-1}}\phi}{{\mathbf 1}_{-\lambda}}_\lambda
 = \int_K\phi(gk)dk \\
 &\quad\quad\quad\quad\quad
 = \int_K \phi(k)P_\lambda(k^{-1}g)dk.
\end{aligned}
\end{equation}
Then it is known that the image of $\mathcal P_\lambda$ is contained in 
$\mathcal A(G/K;\mathcal M_\lambda)$ because $DP_\lambda = \gamma_\lambda(D)P_\lambda$
for $D\in U(\mathfrak g)^K$. 
(If the center $Z$ of $G$ is infinite, integrations over $K$ in the 
definitions of pairing $\langle\,\cdot\,,\cdot\,\rangle_\lambda$ and 
the Poisson transform should be understood to be normalized integral over 
$K/Z$. But we write $K$ for simplicity.)

For $\alpha\in\Sigma(\mathfrak a_{\mathfrak p})$
and $w\in W(\mathfrak a_{\mathfrak p})$, 
we put
\begin{equation}\label{eq:4.15}
\begin{aligned}
 \Sigma(\mathfrak a_{\mathfrak p})_o^+ &= \{\alpha\in\Sigma(\mathfrak a_{\mathfrak p})^+;
 \tfrac\alpha2\notin\Sigma(\mathfrak a_{\mathfrak p})\},\\
 e_\alpha(\lambda) &= 
               \left\{
         \Gamma\left(\tfrac{\lambda_\alpha}4 + \tfrac{m_\alpha}4
                   + \tfrac12\right)
         \Gamma\left(\tfrac{\lambda_\alpha}4 + \tfrac{m_\alpha}4
                   + \tfrac {m_{2\alpha}}2\right)
               \right\}^{-1},\\
 e(\lambda) &= \prod_{\alpha\in\Sigma(\mathfrak a_{\mathfrak p})_o^+}
           e_\alpha(\lambda),
\\
c(\lambda) & =C e(\lambda)
\prod_{\alpha\in\Sigma(\mathfrak a_{\mathfrak p})_o^+}
2^{-\frac{\lambda_\alpha}2}
\Gamma\left(\tfrac{\lambda_\alpha}2\right),
\end{aligned}
\end{equation}
where $m_\alpha$ is the multiplicity of the root $\alpha$, 
$\lambda_\alpha = 2\frac{\ang{\lambda}{\alpha}}{\ang{\alpha}{\alpha}}$ 
and $C$ is a constant determined by $c(\rho)=1$.

The following theorem is the main result in \cite{K--}.
\begin{thm}\label{thm:4.2}
Let $\lambda\in\mathfrak a_{\mathbb C}^*$. 

{\rm i)}
$\mathcal P_\lambda$ gives a topological $G$-isomorphism of 
$\mathcal A(G/P;\mathcal L_\lambda)$ onto $\mathcal A(G/K;\mathcal M_\lambda)$
if and only if $e(\lambda)\ne0$. 

{\rm ii)} 
Let $w$ be an element of $W(\mathfrak{a}_\mathfrak{p})$ 
which satisfies
\begin{equation}\label{eq:4.17}
 \operatorname{Re}\ang{w\lambda}{\alpha}\ge0
 \quad\text{for all }\alpha\in\Sigma(\mathfrak a_{\mathfrak p})^+, 
\end{equation}
then $\mathcal{P}_{w\lambda}$ gives a 
topological $G$-isomorphism
\begin{equation}\label{eq:4.18}
 \mathcal P_{w\lambda}: \mathcal B(G/P;\mathcal L_{w\lambda})
 \ \widetilde\to\ \mathcal A(G/K;\mathcal M_\lambda).
\end{equation}
\end{thm}

\begin{rem}\label{rem:4.3}
{\rm i)}
The equivalence of the injectivity of $\mathcal P_\lambda$ and the condition
$e(\lambda)\ne0$ is proved in \cite{He3}.

{\rm ii)}
Suppose $e(\lambda)\ne0$.
Let $\mathcal D'(G)$ and $C^\infty(G)$ denote the space of distributions and
that of $C^\infty$-functions on $G$, respectively.
Then
\begin{equation}
\begin{aligned}
 &\mathcal P_\lambda(\mathcal B(G/P;\mathcal L_\lambda)\cap\mathcal 
 D'(G))\label{eq:4.19}\\
 &\quad
 = \{u\in\mathcal A(G/K;\mathcal M_\lambda)\,;\,
    \text{there exist $C$ and $k$ with }|u(g)|\le C\exp k|g|
  \},\\
\end{aligned}
\end{equation}
\begin{equation}
\begin{aligned}
 &\mathcal P_\lambda(\mathcal B(G/P;\mathcal L_\lambda)\cap C^\infty(G))\label{eq:4.20}\\
 &\quad
 = \{u\in\mathcal A(G/K;\mathcal M_\lambda)\,;\,
    \text{there exist $k$ such that for any $D\in U(\mathfrak k)$}\\
 &\quad\quad\quad\text{ we can choose
    $C_D>0$ with } |\pi_Du(g)|\le C_D\exp k|g|
  \}.
\end{aligned}
\end{equation}
Here $U(\mathfrak k)$ is the universal enveloping algebra of the complexification
of $\mathfrak k$ and $|g| = \ang HH^{\frac 12}$ with the Killing form 
$\ang{\ }{\ }$ if $g\in K\exp H K$ with $H\in\mathfrak a_{\mathfrak p}$.
Note that $U(\mathfrak k)$ in (\ref{eq:4.20}) may be replaced by $U(\mathfrak g)$.

In fact, \eqref{eq:4.19} is given in \cite[Corollary~5.5]{OS1}.
Suppose $u=\mathcal P_\lambda f$.
Since $P_\lambda$ is contained in the set \eqref{eq:4.19}, the 
$U(\mathfrak g)$-equivariance and the last expression in \eqref{eq:4.14} implies that
the left hand side of \eqref{eq:4.20} is contained in the right hand side of \eqref{eq:4.20}.
Note that the inverse of $\mathcal P_\lambda$ is the map of taking 
boundary values.  We can see from the definition that the order of
distribution of the boundary value $f$ of $u$ is estimated by $k$ in \eqref{eq:4.19} 
(cf.~\cite[the proof of Lemma~2.19]{OS1} or \cite{O3}).
If $u$ is contained in the right hand side of \eqref{eq:4.20},
the order of $\pi_Df$ is uniformly bounded for all $D\in U(\mathfrak k)$ and 
hence $f|_K\in C^\infty(K)$.
A different proof can be found in \cite{BS}.
\end{rem}

{\em Proof of Theorem~\ref{thm:4.1}.} 
Let $X\in\mathfrak g$, $k\in K$ and $u\in\mathcal A(G/K)$.
Then
\begin{align*}
 (X\pi_ku)(e) & = \left.\tfrac d{dt}u\bigl(k^{-1}\exp tX\bigr)\right|_{t=0}
              = \left.\tfrac d{dt}u\bigl((\exp t\Ad(k^{-1})X)k^{-1}\bigr)\right|_{t=0} \\
&              = (\Ad(k)^{-1}Xu)(e)
\end{align*}
and therefore the bilinear from $\ang{\ }{\ }$ in Theorem~\ref{thm:4.1} is 
$K$-invariant.

Let $\hat K$ be the set of equivalence classes of irreducible representations
of $K$.
For $\delta$, $\tau\in\hat K$ we denote by $\mathcal A(G/K;\mathcal M_\lambda)_\delta$
and $\mathcal H_\tau$ the $\delta$ isotopic components of 
$\mathcal A(G/K;\mathcal M_\lambda)$ and $\tau$ isotopic components of $\mathcal H$, 
respectively.
In general, for a $K$-module $U$ we denote by $U_\delta$ the subspace
of $K$-isotopic components $\delta$.
Then the $K$-equivariant map
\[
 \mathcal A(G/K;\mathcal M_\lambda)_\delta\ni u\mapsto
 (\mathcal H_\tau\ni h\mapsto \ang hu\in \mathbb C)\in \mathcal H_{\tau^*}
\]
is identically zero if $\delta\ne\tau^*$ by Schur's lemma, where $\tau^*$ is
the dual of $\tau$.

Suppose $u\in\mathcal A(G/K;\mathcal M_\lambda)_\delta$ satisfies
$\ang hu = 0$ for any $h\in\mathcal H_{\delta^*}$.
Then $\ang hu = 0$ for any $h\in\mathcal H$ and therefore it follows
from \eqref{eq:4.1}, \eqref{eq:4.5} and \eqref{eq:4.6} that $(Du)(e)=0$ for all $D\in U(\mathfrak g)$.
Hence $u=0$ because $u$ is real analytic.
On the other hand, since $\mathcal H$ and $\mathcal A(G/K;\mathcal M_\lambda)_K$ are
isomorphic to $\Ind_M^K\mathbf 1$ (cf.~\cite{KR} and Theorem~\ref{thm:4.2})
$\dim \mathcal A(G/K;\mathcal M_\lambda)_\delta = \dim \mathcal H_{\delta^*}$
and hence we can conclude that $\ang{\ }{\ }$ defines a non-degenerate 
bilinear form on $\mathcal A(G/K;\mathcal M_\lambda)_\delta \times 
\mathcal H_{\delta^*}$ and we have i) and ii).
Here we remark that the results follows from the weaker relation
$\dim\mathcal A(G/K;\mathcal M_\lambda)_\delta \ge \dim\mathcal H_{\delta^*}$.

First note that the map \eqref{eq:4.9} is a bijection of $\mathfrak V(\lambda)$ onto
$\mathfrak V(\lambda)_K$ whose inverse is the map of taking the closure in 
$C^\infty(G)$.
The map is still bijective even if it is restricted to the spaces killed
by a left ideal $J$ of $U(\mathfrak g)$.
Moreover remark that for $X\in\mathfrak g$, $D\in U(\mathfrak g)$ and 
$u\in\mathcal A(G/K)$ we have $(D\pi_Xu)(e) = -(XDu)(e)$.

Let $V_K\in\mathfrak V(\lambda)_K$.
Then $(D\sm(h)u)(e)=0$ for $D\in U(\mathfrak g)$, $h\in H(V_K)$ and 
$u\in V_K$ because of the above remark.
Note that for a left ideal $J$ of $U(\mathfrak g)$ and a function $u$ in
$\mathcal A(G)$, the condition $Du=0$ for all $D\in J$ is equivalent to
the condition $(Du)(e)=0$ for all $D\in J$.
Hence we have
\begin{align*}
 V_K &= \{u\in\mathcal A(G/K;\mathcal M_\lambda)_K; \ang hu = 0
      \quad\text{for }h\in H(V_K)\}\\
   &= \{u\in\mathcal A(G/K;\mathcal M_\lambda)_K; (D\sm(h)u)(e) = 0
      \quad\text{for }h\in H(V_K)\text{ and }D\in U(\mathfrak g)\}\\
   &= \{u\in\mathcal A(G)_K; Du = 0
      \quad\text{for }D\in J_\lambda + \sum_{h\in H(V_K)}U(\mathfrak g)\sm(h)
      \}.
\end{align*}

Let $J$ be a left ideal of $U(\mathfrak g)$ containing $J_\lambda$.
Then \eqref{eq:4.1} and \eqref{eq:4.5} show that 
\begin{equation}\label{eq:4.21}
\begin{aligned}
  J &= J_\lambda \oplus \{\sm(h); h\in\mathcal H_J\}\quad\text{with}\\
  \mathcal H_J &= \sm^{-1}(J)\cap\mathcal H
\end{aligned}
\end{equation}
and
\begin{align*}
\{u\in\mathcal A(G)_K\,;\, & Du = 0\quad  \text{for \ }D\in J\} 
\\ & =
\{u\in\mathcal A(G/K;\mathcal M_\lambda)_K; \sm(h)u = 0\quad\text{for \ }h
\in\mathcal H_J\}.
\end{align*}
Hence the map of $\mathfrak J(\lambda)$ to $\mathfrak V(\lambda)_K$ is injective
and we have the Theorem~\ref{thm:4.1}.
\qed

\begin{thm}\label{thm:4.4}
{\rm i)}
The map
\begin{equation}\label{eq:4.22}
\mathcal H\ni h\ \mapsto \pi_{\sm(h)}{\mathbf 1}_\lambda\in
\mathcal A(G/P;\mathcal L_{\lambda})_K
\end{equation}
is $K$-linear. 
It is bijective if and only if $e(-\lambda)\ne0$.

{\rm ii)}
\begin{equation}\label{eq:4.23}
 \gamma_\lambda(D) = \pi_D({\mathbf 1}_{-\lambda})(e)
 \quad\text{for \ }D\in U(\mathfrak g).
\end{equation}

{\rm iii)}
Putting
\begin{equation}\label{eq:4.24}
 \mathcal H_\lambda = \{h\in\mathcal H; \gamma_\lambda(\sm(\Ad(k)h)) = 0
 \quad\text{for all }k\in K\}
\end{equation}
and 
\begin{equation}\label{eq:4.25}
 \bar J_\lambda =  J_\lambda + \sum_{h\in\mathcal H_\lambda}U(\mathfrak g)\sm(h),
\end{equation}
we have
\begin{equation}\label{eq:4.26}
\operatorname{Im}\mathcal P_\lambda = \{u\in\mathcal A(G);\ Du=0\quad\text{for }
D\in\bar J_\lambda\}.
\end{equation}
\end{thm}
\begin{pf}
Since ${\mathbf 1}_\lambda$ is $K$-invariant, the map \eqref{eq:4.22} is 
$K$-equivariant.
Moreover for $h\in\mathcal H$, the condition $\pi_{\sm(h)}{\mathbf 1}_\lambda=0$ is 
equivalent to $(\pi_{\sm(\Ad(k)h)}{\mathbf 1}_\lambda)(e) = 0$ for $k\in K$ because
$(\pi_{\sm(h)}{\mathbf 1}_\lambda)(kan) = 
(\pi_{\sm(h)}{\mathbf 1}_\lambda)(k)a^{\rho - \lambda}$.
On the other hand, \eqref{eq:4.23} follows from the definition of $\gamma_\lambda$ and
${\mathbf 1}_{-\lambda}$.

Let $h\in\mathcal H_\lambda$.
Then $\pi_{\sm(h)}{\mathbf 1}_{-\lambda} = 0$ and therefore 
$\sm(h)P_\lambda=0$ and hence it is clear from \eqref{eq:4.14} that 
$\operatorname{Im}\mathcal P_\lambda\subset
 \{u\in\mathcal A(G); Du =0\quad\text{for \ }D\in\bar{J_\lambda}\}.
$

Since $\pi_D{\mathbf 1}_\lambda\in\mathbb C{\mathbf 1}_\lambda$ for 
$D\in U(\mathfrak g)^K$, \eqref{eq:4.1}
shows
\[
 U(\mathfrak g){\mathbf 1}_\lambda = \{\pi_{\sm(h)}{\mathbf 1}_\lambda; 
h\in\mathcal H\},
\]
which is the Harish-Chandra module of the minimal closed $G$-invariant
subspace of $\mathcal A(G/P;\mathcal L_\lambda)$ containing ${\mathbf 1}_\lambda$.
For $\phi\in\mathcal A(G/P;\mathcal L_\lambda)_\delta$, we have
$\mathcal P_\lambda\phi(g) = \ang{\phi}{\pi_g{\mathbf 1}_{-\lambda}}_\lambda$
and therefore the condition $\mathcal P_\lambda\phi=0$ is equivalent
to $\ang\phi{\pi_{\sm(h)}{\mathbf 1}_\lambda}_\lambda=0$ for all 
$h\in\mathcal H_{\delta^*}$.
Hence
$
 [\operatorname{Ker}\mathcal P_\lambda:\delta] = 
 [\mathcal H_\lambda:\delta^*]
$
and Theorem~\ref{thm:4.2} shows
$
 [\operatorname{Im}\mathcal P_\lambda:\delta]
 = [\mathcal A(G/K;\mathcal M_\lambda):\delta] - [\mathcal H_\lambda:\delta^*],
$
which means $\dim(\operatorname{Im}\mathcal P_\lambda)_\delta 
 = \dim\{u\in\mathcal A(G);\ Du=0\quad\text{for }D\in\bar J_\lambda\}_\delta$
and furthermore \eqref{eq:4.26} owing to Theorem~\ref{thm:4.1} and Remark~\ref{rem:4.3}~i).

If $e(-\lambda)\ne0$, the bijectivity of \eqref{eq:4.22} follows from Theorem~\ref{thm:4.2}
because $\mathcal H_\lambda=\{0\}$ by the argument above.
But it follows directly from the result in \cite{Ko} (cf.~\cite{He3}) that 
${\mathbf 1}_\lambda$ is cyclic in $\mathcal A(G/K;\mathcal L_\lambda)$ if and only if 
$e(-\lambda)\ne0$.
\qed
\end{pf}
\begin{rem}[cf.~\cite{Du}]
When $\mathfrak g$ is a complex semisimple Lie algebra, $J_\lambda$ is identified
with the annihilator of the Verma module of $\mathfrak g$ parametrized by $\lambda$
and the Poisson transform gives a bijection between the two sided ideals of the
universal enveloping algebra of $\mathfrak g$ containing $J_\lambda$ 
which are identified with $\mathfrak J(\lambda)$ and the closed 
$G$-invariant subspaces of class 1 principal series of $G$ parametrized
by $\lambda$.
\end{rem}

For a subset $\Xi$ of $\Psi(\mathfrak a_{\mathfrak p})$, let $W_\Xi$ be a 
subgroup of
$W(\mathfrak a_{\mathfrak p})$ generated by reflections with respect to the elements of $\Xi$
and put $P_\Xi = PW_\Xi P$.
Let $P_\Xi = M_\Xi A_\Xi N_\Xi$ be the Langlands 
decomposition of $P_\Xi$ with $A_\Xi\subset A$.
For an element $\mu$ of the complex dual $\mathfrak a_{\Xi,\mathbb C}^*$ of
the Lie algebra $\mathfrak a_\Xi$ of $A_\Xi$,  the space of  hyperfunction
sections of spherical degenerate series is defined by
\begin{multline}\label{eq:4.27}
 \mathcal B(G/P_\Xi;\mathcal L_{\Xi, \mu})
 = \{f\in\mathcal B(G); f(gman) = a^{ \mu-\rho }f(g)\\
   \text{for }(g,m,a,n)\in G\times M_\Xi\times A_\Xi\times N_\Xi\}.
\end{multline}
Then as in the case of the minimal parabolic subgroup, we can define Poisson
transform
\begin{equation}\label{eq:4.28}
\begin{aligned}
 \mathcal P_{\Xi,\mu}: \mathcal B(G/P_\Xi;\mathcal L_{\Xi, \mu})
 \ &\rightarrow\ \mathcal B(G)\\
 \phi\ &\mapsto\ 
 (\mathcal P_{\Xi,\mu}\phi)(g) =
  \ang{\pi_{g^{-1}}\phi}{{\mathbf 1}_{\Xi,-\mu}}_{\Xi,\mu}\\
 &\quad\quad\quad\quad
 =\int_K\phi(gk)dk = \int_K\phi(k)P_{\Xi,\mu}(k^{-1}g)dk.
\end{aligned}
\end{equation}
Here $\ang{\ }{\ }_{\Xi,\mu}$ is the bilinear form of 
$\mathcal B(G/P_\Xi;\mathcal L_{\Xi,\mu})\times
\mathcal A(G/P_\Xi;\mathcal L_{\Xi,-\mu})$ defined on the integral over $K$,
${\mathbf 1}_{\Xi,\mu}(kman) = a^{\mu-\rho}$ for $(k,m,a,n)\in K\times
M_\Xi\times A_\Xi\times N_\Xi$ and
$P_{\Xi,\mu}(g) = {\mathbf 1}_{\Xi,-\mu}(g^{-1})$.

Now we remark
\begin{lem}\label{lem:4.5}
We have naturally
\begin{align}
 \mathcal B(G/P_\Xi;\mathcal L_{\Xi,\mu}) &\subset
 \mathcal B(G/P;\mathcal L_{\mu + \rho(\Xi)}),\label{eq:4.29}\\
 \operatorname{Im}\mathcal P_{\Xi,\mu} &=
 \operatorname{Im}\mathcal P_{\mu - \rho(\Xi)}.\label{eq:4.30}
\end{align}
Here we identify $\mathfrak a_{\mathfrak p}$ with its dual by the Killing form 
and regard $\mu\in\mathfrak{a}_{\Xi,\mathbb{C}}^*$ as 
an element of $\mathfrak{a}_\mathbb{C}^*$ with value zero 
on $\mathfrak{a}_\Xi^{\perp}$, and
define $\rho_\Xi = \rho|_{\mathfrak a_\Xi}$ and 
$\rho(\Xi) = \rho - \rho_\Xi$.
\end{lem}
\begin{pf}
The inclusion \eqref{eq:4.29} is clear from \eqref{eq:4.12} and \eqref{eq:4.27}, which implies
${\mathbf 1}_{\Xi,-\mu} = {\mathbf 1}_{-(\mu-\rho(\Xi))}$ because 
$\mathcal A(G/P_\Xi,\mathcal L_{\Xi,-\mu})\subset\mathcal A(G/P;\mathcal L_{-\mu+\rho(\Xi)})$.
Since $P_{\Xi,\mu}$ is left $M_\Xi$-invariant and
\[
 \mathcal B(G/P_\Xi;\mathcal L_{\Xi,\mu})|_K = \mathcal B(K/M_\Xi\cap K)\text{ and }
 \mathcal B(G/P;\mathcal L_{\mu-\rho(\Xi)})|_K = \mathcal B(K/M),
\]
we have \eqref{eq:4.30} from \eqref{eq:4.14} and \eqref{eq:4.18}.
\qed
\end{pf}

\begin{cor}\label{cor:4.6}
{\rm i)}
$\mathcal P_{\Xi,\mu}$ is injective if
$e(\mu + \rho(\Xi))\ne0$.
In particular, the Poisson transform 
$\mathcal P_{\rho_\Xi}:\mathcal B(G/P_\Xi)\rightarrow\mathcal A(G/K;\mathcal M_\rho)$
is injective.

{\rm ii)}
If $e(-\mu+\rho(\Xi))e(\mu+\rho(\Xi))\ne0$, then
$\mathcal B(G/P_\Xi;\mathcal L_{\Xi,\mu})$ is irreducible.
\end{cor}
\begin{pf}
The claim i) is a direct consequence of Theorem~\ref{thm:4.2}~i) and 
Lemma~\ref{lem:4.5}.

The $K$-invariant bilinear form $\ang{\ }{\ }_{\Xi,\mu}$ and \eqref{eq:4.28} show
that the following statements are equivalent:
\begin{align}
 &\mathcal P_{\Xi,\mu}\text{ is injective.}\label{eq:4.31}\\
 &{\mathbf 1}_{\Xi,-\mu}\text{ is cyclic in }\mathcal B(G/P_\Xi;\mathcal L_{\Xi,-\mu}).
 \label{eq:4.32}\\
 &\text{Any non-zero closed $G$-invariant subspace of}\label{eq:4.33}\\
 &\mathcal B(G/P_\Xi;\mathcal L_{\Xi,\mu})
 \text{ contains }{\mathbf 1}_{\Xi,\mu}.\notag
\end{align}
Hence ii) is clear.
\qed
\end{pf}

\begin{rem}\label{rem:4.7}
{\rm i)}
The calculation of $\mathcal H_\lambda$ in \eqref{eq:4.24} is equivalent to the 
determination of the kernel of $P^\gamma(\lambda)$ defined by \cite[\S4]{Ko}.

{\rm ii)} Under the notation in Lemma~\ref{lem:4.5}
\begin{equation}\label{eq:4.34}
  [\mathcal H_{\mu+\rho(\Xi)}:\delta] \ge [\Ind_M^K\mathbf1:\delta]
  - [\Ind_{M_\Xi(K)}^K\mathbf1:\delta]
  \quad\text{for }\delta\in\hat K
\end{equation}
and the equality holds if and only if $\mathcal P_{\Xi,\mu}$ is injective.
\end{rem}
\medskip

Most of statements in this section can be generalized to line bundles
or vector bundles over $G/K$.
We will give necessary modifications when we consider homogeneous line 
bundles over a Hermitian symmetric space $G/K$.
For simplicity we suppose $G$ is simple and $\mathfrak k$ have a non-trivial 
center.
Let $K'$ be the analytic subgroup of $G$ with the Lie algebra 
$\mathfrak k'= [\mathfrak k,\mathfrak k]$ and $Y$ be the central element of $\mathfrak k$ 
 normalized so that $\exp tY\in K'_R$ if and only if $\ell\in\mathbb Z$,
where $K'$ is denoted by $K'_R$ when $G$ is a real form of a simply connected
complex Lie group  (cf.~\cite{Sn1}).
Let $\chi_\ell: K\to\mathbb C$ be the one dimensional representation of $K$
defined by $\chi_\ell(k)=1$ if $k\in K'$ and 
$\chi_\ell(\exp tY)=\exp\sqrt{-1}\ell t$.
Then we can define the space of real analytic sections of a homogeneous line 
bundle ${\mathbf E}_\ell$ over $G/K$ associated to the representation 
$\chi_\ell$ of $K$:

\begin{equation}\label{eq:4.35}
 \mathcal A(G/K;{\mathbf E}_\ell) =
 \{u\in\mathcal A(G); u(gk) = \chi_\ell(k)^{-1}u(g)\text{ for }k\in K\}.
\end{equation}
Let $\mathbb D({\mathbf E}_\ell)$ be the algebra of invariant differential operators 
acting on sections of ${\mathbf E}_\ell$.
Defining $\gamma_\ell(D)\in U(\mathfrak a)$ for $D\in U(\mathfrak a)$ so that
\begin{equation}\label{eq:4.36}
 e^\rho\scirc \gamma_\ell(D) \scirc e^{-\rho}
  \in\mathfrak n U(\mathfrak n + \mathfrak a_{\mathfrak p})
     \oplus \textstyle\sum_{X\in\mathfrak k}U(\mathfrak g)(X+\chi_\ell(X)),
\end{equation}
as in the case of $\gamma$, we have the Harish-Chandra isomorphism
\begin{equation}\label{eq:4.37}
  \bar{\gamma_\ell}:\mathbb D({\mathbf E}_\ell)\simeq
  U(\mathfrak g)^K/(U(\mathfrak g)^K\cap
  \textstyle\sum_{X\in\mathfrak k}U(\mathfrak g)(X+\chi_\ell(X)))
  \to U(\mathfrak a_{\mathfrak p})^W
\end{equation}
onto $U(\mathfrak a_{\mathfrak p})^W$.
Hence we put
\begin{equation}\label{eq:4.38}
  J^\ell_\lambda = \textstyle\sum_{X\in\mathfrak k}U(\mathfrak g)(X + \chi_\ell(X))
                  + \textstyle\sum_{p\in\mathcal O(\mathfrak p)^K}
U(\mathfrak{g})(\sm(p)
                       - \gamma^\ell_\lambda(\sm(p)))
\end{equation}
and
\begin{equation}\label{eq:4.39}
  \mathcal A(G/K;\mathcal M^\ell_\lambda) = 
  \{u\in\mathcal A(G); Du=0\quad\text{for }D\in J^\ell_\lambda\}.
\end{equation}
\begin{thm}\label{thm:4.8}
Replacing $\mathcal M_\lambda$ by $\mathcal M^\ell_\lambda$,
the statements\/ {\rm i), ii)} and\/ {\rm iii)} in\/ {\rm Theorem~\ref{thm:4.1}} are 
valid if the $K$-invariantness of\/ $\ang{\ }{\ }$ is modified by
\begin{equation}\label{eq:4.40}
 \ang{\Ad(k)h}{\chi_{-\ell}(k)\pi_ku} = \ang hp
 \quad\text{for }k\in K, h\in\mathcal H
 \text{ and }u\in\mathcal A(G/K;\mathcal M^\ell_\lambda).
\end{equation}
\end{thm}
\begin{pf}
Recalling the proof of Theorem~\ref{thm:4.1}, we have only to consider
\begin{equation}\label{eq:4.41}
 \mathcal A(G/K;\mathcal M^\ell_\lambda)_{\delta\otimes\chi_\ell}\times
 \mathcal H_{\delta^*}\ni (u,p)\mapsto\ang pu\in\mathbb C
\end{equation}
because of the $K$-invariance \eqref{eq:4.40}.
Hence if we have 
$\dim A(G/K;\mathcal M^\ell_\lambda)_{\delta\otimes\chi_\ell}\ge\dim
\mathcal H_{\delta^*}$, the same argument as in the proof of Theorem~\ref{thm:4.1}
shows Theorem~\ref{thm:4.8}.

On the other hand, the proof of \cite[Lemma~8.6] {Sn1} says that 
\[
 [A(G/K;\mathcal M^\ell_\lambda),\delta\otimes\chi_\ell]
  \ge[\Ind_M^K \chi_\ell|_M,\delta\otimes\chi_\ell].
\]
Tensoring $\chi_{-\ell}$ to the right hand side, we have
$[\Ind_M^K \chi_\ell|_M,\delta\otimes\chi_\ell] = [\Ind_M^K{\mathbf 1},\delta]$,
which is also equal to $[\mathcal H,\delta^*]$.
\qed
\end{pf}

\bigskip
Put
$\{|\alpha|;\alpha\in\Sigma(\mathfrak a_{\mathfrak p})\} = \{c_1,\ldots,c_N\}$
with $c_1>\cdots>c_N$.
Then $N=1$ or 2 or 3 and we fix $\beta_\nu\in\Sigma(\mathfrak a_{\mathfrak p})^+$
with $|\beta_\nu|=c_\nu$.
Moreover we put
\begin{multline}\label{eq:4.42}
 \mathcal B(G/P;\mathcal L^\ell_\lambda)
 = \{f\in\mathcal B(G);
     f(gman) = \chi_\ell(m)^{-1}a^{\lambda-\rho}f(g)\\
 \text{for }(g,m,a,n)\in G\times M\times A\times N\}
\end{multline}
and
\begin{equation}\label{eq:4.43}
\begin{aligned}
  e_\alpha(\lambda,\ell) &= 
               \left\{
         \Gamma\left(\tfrac{\lambda_\alpha}2
                    +\tfrac{m_{\frac\alpha2}}4
                    +\tfrac{1+\ell}2\right)
         \Gamma\left(\tfrac{\lambda_\alpha}2
                    +\tfrac{m_{\frac\alpha2}}4
                    +\tfrac{1-\ell}2\right)
               \right\}^{-1},\\
 e(\lambda,\ell) &= 
           \prod_{\alpha\in\Sigma(\mathfrak a_{\mathfrak p})^+,\ |\alpha|=|\beta_1|}
             e_\alpha(\lambda,\ell)
           \times
           \prod_{\alpha\in\Sigma(\mathfrak a_{\mathfrak p})^+,\ |\alpha|=|\beta_2|}
             e_\alpha(\lambda),\\
 c(\lambda,\ell) &=
           C e(\lambda,\ell)
           \prod_{k=1}^2 \prod_{\alpha\in\Sigma(\mathfrak a_{\mathfrak p})^+,\
             |\alpha|=|\beta_k|}2^{-\frac{\lambda_\alpha}{k}}
           \Gamma\left(\tfrac{\lambda_\alpha}{k}\right),
\end{aligned}
\end{equation}
where $C\in\mathbb R$ is determined by $c(\rho,0)=1$,
$e_\alpha(\lambda)$ is given in (\ref{eq:4.15}) and
$m_{\frac\alpha2}$ may be $0$.
Then the main result in \cite{Sn1} says

\begin{thm}\label{thm:4.9}
{\rm i)}
The Poisson transform
\[
\begin{aligned}
 \mathcal P^\ell_\lambda:
  \mathcal B(G/P;\mathcal L^\ell_\lambda)&\to\mathcal B(G)\\
  \phi&\mapsto
  (\mathcal P^\ell_\lambda\phi)(g) = \int_K \phi(gk)\chi_\ell(k)dk
  =\int_K\phi(k){\mathbf 1}_{-\lambda,-\ell}(g^{-1}k)dk
\end{aligned}
\]
is a $G$-homomorphism and\/ 
$\operatorname{Im}\mathcal P^\ell_\lambda\subset\mathcal A(G/K;\mathcal M^\ell_\lambda)$.
Here the function ${\mathbf 1}_{-\lambda,-\ell}\in\mathcal A(G/P;\mathcal L^{-\ell}_{-\lambda})$ 
is defined by
${\mathbf 1}_{-\lambda,-\ell}(kan) = \chi_\ell(k)a^{-\lambda-\rho}$
for $(k,a,n)\in K\times A\times N$.

{\rm ii)} If
\begin{equation}\label{eq:4.44}
 -2\frac{\ang\lambda\alpha}{\ang\alpha\alpha}
 \notin\{1,2,3,\ldots\}
 \quad\text{for }\alpha\in\Sigma(\mathfrak a_{\mathfrak p})^+
\end{equation}
and $e(\lambda,\ell)\ne 0$, then $\mathcal P^\ell_\lambda$ is a topological
$G$-isomorphism of $\mathcal B(G/P;\mathcal L^\ell_\lambda)$ onto
$\mathcal A(G/K;\mathcal M^\ell_\lambda)$.

{\rm iii)} Suppose $\operatorname{Re}\ang\lambda\alpha>0$ for 
$\alpha\in\Sigma(\mathfrak a_\mathfrak p)^+$.
Then the following statements are equivalent:
\begin{align}
 &c(\lambda,\ell)\ne 0.\label{eq:4.45}\\
 &\mathcal P^\ell_\lambda\text{ is injective.}\label{eq:4.46}\\
 &\operatorname{Im}\mathcal P^\ell_\lambda= \mathcal A(G/K;\mathcal M^\ell_\lambda).
\end{align}
\end{thm}

Notice that Theorem~\ref{thm:4.2} (ii) does not hold in the case 
of non-trivial line bundles. 

Now we consider the degenerate series.
Let $M_{\Xi,s}^o$ denote the semisimple part of $M_{\Xi}$, namely
$M_{\Xi,s}^o$ is the analytic subgroup of $M_{\Xi}$ with the Lie algebra
$[\mathfrak m_{\Xi},\mathfrak m_{\Xi}]$.
Suppose
\begin{equation}\label{eq:4.48}
 \chi_\ell|_{M_{\Xi,s}^o\cap K} = \mathbf 1.
\end{equation}
Let $Z(M_\Xi)$ be the center of $M_\Xi$.
Then $Z(M_\Xi)\subset M$ and for $\mu\in\mathfrak a_\Xi^*$ 
we can define a one-dimensional representation 
$\tau_{\ell,\mu}$ of $P_\Xi$ by
\begin{equation}\label{eq:4.49}
 \tau_{\Xi,\ell,\mu}(yman) = \chi_\ell(m)a^{\mu-\rho}
 \quad\text{for }(y,m,a,n)\in 
 M_{\Xi,s}^o\times M\times A_\Xi\times N_\Xi.
\end{equation}
Put
\begin{equation}\label{eq:4.50}
\begin{aligned}
 \mathcal B(G/P_\Xi;\mathcal L^\ell_{\Xi,\mu})
  &= \{f\in\mathcal B(G); f(gp)=\tau_{\Xi,\ell,\mu}(p)f(g)
     \quad\text{for }p\in P_\Xi\},\\
 \mathcal A(G/P_\Xi;\mathcal L^\ell_{\Xi,\mu})
  &= \mathcal B(G/P_\Xi;\mathcal L^\ell_{\Xi,\mu}) \cap \mathcal A(G).
\end{aligned}
\end{equation}
and define an element ${\mathbf 1}_{\Xi,\ell,\mu}\in
\mathcal A(G/P_\Xi;\mathcal L^\ell_{\Xi,\mu})$
by
\begin{equation}\label{eq:4.51}
  {\mathbf 1}_{\Xi,\ell,\mu}(kan) = \chi_{-\ell}(k)a^{\mu-\rho}
  \quad\text{for }(k,a,n)\in K\times A\times N.
\end{equation}
Then by the $G$-invariant bilinear form
\begin{equation}\label{eq:4.52}
 \mathcal B(G/P_\Xi;\mathcal L^\ell_{\Xi,\mu}) \times 
 \mathcal A(G/P_\Xi;\mathcal L^{-\ell}_{\Xi,-\mu})
 \ni (\phi,f)\mapsto\ang\phi{f}_{\Xi,\ell,\mu} = \int_K \phi(k)f(k)dk
 \in\mathbb C
\end{equation}
we can define the Poisson transform
\begin{equation}\label{eq:4.53}
\begin{aligned}
 \mathcal P^\ell_{\Xi,\mu}:
   \mathcal B(G/P_\Xi;\mathcal L^\ell_{\Xi,\mu})&\to
   \mathcal A(G/K;\mathcal M^\ell_{\mu + \rho(\Xi)})\\
 \phi&\mapsto
 (P^\ell_{\Xi,\mu}\phi)(g) =
 \ang{\pi_{g^{-1}}\phi}{{\mathbf 1}_{\Xi,-\mu,-\ell}}_{\Xi,\ell,\mu}\\
 &\quad\quad\quad\quad =
 \ang{\phi}{\pi_g{\mathbf 1}_{\Xi,-\mu,-\ell}}_{\Xi,\ell,\mu}.
\end{aligned}
\end{equation}
We note the following lemma which is similarly proved as Lemma~\ref{lem:4.5}.

\begin{lem}\label{lem:4.10}
\begin{align}
 \mathcal B(G/P_\Xi;\mathcal L^\ell_{\Xi,\mu}) &\subset
 \mathcal B(G/P;\mathcal L^\ell_{\mu + \rho(\Xi)}),\label{eq:4.54}\\
 \operatorname{Im}\mathcal P_{\Xi,\mu}^\ell &=
 \operatorname{Im}\mathcal P_{\mu - \rho(\Xi)}^\ell.\label{eq:4.55}
\end{align}
\end{lem}

Lastly in this section we examine the space of harmonic
functions on $G/K$:
\[
  \mathcal H(G/K):=\mathcal A(G/K;\mathcal M_\rho).
\]
Let $X_\Xi$ be the Satake compactification of $G/K$ where the boundary $G/P_\Xi$
appears.
For $a\in A$ we denote by $a\to\infty$ if $\alpha(\log a)\to\infty$
for any $\alpha\in\Psi(\mathfrak a_{\mathfrak p})$.
Then for any $k\in K$ the point $kaK\in G/K\subset X_\Xi$ 
converges to a point in $G/P_\Xi\subset X_\Xi$.
The Poisson transform $\mathcal P_\rho$ defines a bijective homomorphism 
of $\mathcal B(G/P)$ onto $\mathcal H(G/K)$.
Let $\mathcal C(G/P_\Xi)$ be the space of continuous functions on $G/P_\Xi$. 
Note that $\mathcal C(G/P_\Xi)\simeq\mathcal C(K/M_\Xi)
\subset \mathcal C(G/P)\simeq C(K/M)$.

\begin{prop}
Let $\mathcal F$ be $\mathcal C$ or $\mathcal C^m$ or $\mathcal C^\infty$ or 
$\mathcal D'$ or $\mathcal B$.
Note that $\mathcal F(G/P_\Xi)\simeq\mathcal F(K/K\cap M_\Xi)
\subset \mathcal F(G/P)\simeq \mathcal F(K/M)$.
Then we have
\begin{equation}
 \begin{split}
  &\mathcal P_{\Xi,\rho}\mathcal F(G/P_\Xi)\\
  &\ =\{u\in\mathcal H(G/K)\,;\, u(ka)
  \text{ uniformly converges to a continuous function}\\
  &\qquad\text{on $K/K\cap M_{\Xi}$ in the strong topology of $\mathcal F(K)$
   when $a\to\infty$}\}.
 \end{split}
\end{equation}
\end{prop}
This is shown as follows.
Suppose $u$ is a function in the above left hand side.
Then the boundary value $\beta u$ of $u$ equals $\lim_{a\to\infty}u(ka)$ 
(cf.~for example, \cite{OS1} and \cite{BOS}) and $\mathcal P_\rho\beta u=u$.
The assumption implies $\beta u\in\mathcal F(G/P_\Xi)$ and hence  
$\mathcal P_\rho\beta u=\mathcal P_{\Xi,\rho}\beta u$.
Moreover the Poisson transform of the function 
$f\in\mathcal F(G/P_\Xi)\subset \mathcal F(G/P)$
has a limit $\lim_{a\to \infty}(\mathcal P_\rho u)(ka)=u(ka)$
in the strong topology of $\mathcal F(K)$.

In particular if $u\in\mathcal H(G/K)$ can be continuously extended to the 
boundary $G/P_\Xi$ in $X_\Xi$, 
$u$ satisfies many differential equations corresponding to $\Xi$ because it is
in the image of $\mathcal P_{\Xi,\rho}$.

These statements can be extended for general eigenspaces
$\mathcal A(G/K;\mathcal M_\lambda)$ by using weighted boundary values 
given in \cite[Theorem~3.2]{BOS}.
\section{Construction of the Hua type operators}
\subsection{Two sided ideals}
We want to study a good generator system characterizing the image of the 
Poisson transform $\mathcal P_{\Xi,\mu}^\ell$ given by \eqref{eq:4.53}.

The image of the Poisson transform $\mathcal P_\lambda$ given by
\eqref{eq:4.14} is characterized as a simultaneous eigenspace of the invariant 
differential operators $\mathbb D(G/K)$ on the symmetric space $G/K$ when 
the Poisson transform is injective.
The image is also a simultaneous eigenspace of the center $Z(\mathfrak g)$ of 
$U(\mathfrak g)$ with eigenvalues corresponding to the infinitesimal character
and in most cases the system of the equations on $G/K$ defined by the generators
of $Z(\mathfrak g)$ is equals to that defined by $\mathbb D(G/K)$.
In fact, this holds if the image of $Z(\mathfrak g)\subset U(\mathfrak g)^K$ 
under the identification \eqref{def:bargamma} generates
$\mathbb D(G/K)$, which is valid when $G$ is of classical type.
Moreover even when the image of $Z(\mathfrak g)$ doesn't generates 
$\mathbb D(G/K)$, the system of the equations defined by the generators
of $Z(\mathfrak g)$ characterize the image of the Poisson transform \eqref{eq:4.14}
for generic parameter $\lambda$, which follows from \cite{He4} or \cite{Oc}.

Hence we can expect that the system of differential equations characterizing the 
image of the Poisson transform $\mathcal P_{\Xi,\mu}^\ell$ is given by a two 
sided ideal of $U(\mathfrak g)$ at least when the parameter $\mu$ is generic.
In another word it is expected to coincide with the system defined by a 
certain left ideal of $U(\mathfrak g)$ studied in the previous section.
Note that if $\mathcal P_{\Xi,\mu}^\ell$ is injective, the two sided ideal
should kill the preimage $\mathcal B(G/P_\Xi;\mathcal L_{\Xi, \mu}^\ell)$.
Hence the system obtained by the operators killing 
$\mathcal B(G/P_\Xi;\mathcal L_{\Xi, \mu}^\ell)$ is expected to be the desired one.
The annihilator of $\mathcal B(G/P_\Xi;\mathcal L_{\Xi, \mu}^\ell)$ corresponds 
to that of a generalized Verma module, which will be explained.

Let $\pi$  denote the left regular representation of $G$ on 
$\mathcal{F}(G)$ defined by 
\[
(\pi(g)\varphi)(x)=\varphi(g^{-1}x)\quad (x\in G,\,\varphi\in \mathcal{F}(G)). 
\]
Here $\mathcal{F}$ denotes one of function spaces 
such as $\mathcal{A}$ (real analytic functions), $C^\infty$ (smooth functions), 
$\mathcal{D}'$ 
(distributions), 
$\mathcal{B}$ (hyperfunctions).  
The corresponding representation of $\mathfrak{g}$ is denoted by $\pi$. That is
\[
(\pi(X)\varphi)(x)=\left.\tfrac{d}{dt}\varphi(e^{-tX}x)\right|_{t=0}
\quad (\varphi\in C^\infty(G),\,X\in\mathfrak{g},\,x\in G).
\]
For an element $X\in\mathfrak{g}$, let $L_X$ denote the differential 
operator on $G$ defined by 
\[
(L_X\varphi)(x)=\left.\tfrac{d}{dt}\varphi(x e^{tX})\right|_{t=0}
\quad (\varphi\in C^\infty(G),\,X\in\mathfrak{g},\,x\in G).
\]
The universal enveloping algebra $U(\mathfrak{g})$ is 
identified with the algebra of left $G$-invariant differential operators 
on $G$ by the correspondence $X\mapsto L_X$. 
Under this identification we will write $X\varphi=L_X\varphi$  
and denote by $\ord D$ the order of the corresponding differential operator
for $D\in U(\mathfrak g)$.

Let $\mathfrak a$ be the Cartan subalgebra of the complexification of 
$\mathfrak g$ containing $\mathfrak a_{\mathfrak p}$ and let 
$\Sigma(\mathfrak a)^+$ be a compatible positive system of the complexification
attached to the Cartan subalgebra $\mathfrak a$ and let $\mathfrak b$ be
the corresponding Borel subalgebra.
Denoting the fundamental system of $\Sigma(\mathfrak a)^+$ by $\Psi(\mathfrak a)$,
we have 
$\Psi(\mathfrak a_{\mathfrak p})=\{\alpha|_{\mathfrak a_{\mathfrak p}}\,;\,
\alpha\in\Psi(\mathfrak a)\}\setminus\{0\}$.
For a subset $\Xi\subset\Psi(\mathfrak a_{\mathfrak p})$ we define a subset
\[
  \Theta=\{\alpha\in\Psi(\mathfrak a)\,;\,\alpha|_{\mathfrak a_{\mathfrak p}}\in
  \Xi\cup\{0\}\}\subset\Psi(\mathfrak{a})
\]
and let denote by $\mathfrak p_{\Theta}$, $\mathfrak g_\Theta$, 
$\mathfrak n_\Theta$ and $\mathfrak p_0$ 
the complexifications of $\mathfrak p_{\Xi}$,  
$\mathfrak m_{\Xi}+\mathfrak a_{\Xi}$, $\mathfrak n_\Xi$ and the Lie algebra of 
$P$, respectively.
Note that $\Theta$ corresponds to a fundamental system of the root system of
$\mathfrak g_\Theta$.
Let $\lambda$ denote the character of $\mathfrak p_{\Theta}$ defined by
\begin{equation}
 \tau_{\Xi.\ell,\mu}(e^X)=e^{\lambda(X)}\quad(X\in\mathfrak p_{\Xi}).
\end{equation}
Let $\mathfrak a_{\Theta}$ be the center of $\mathfrak g_\Theta$.
Then $\lambda$ is identified with an element of the dual 
$\mathfrak a_{\Theta}^*$ of $\mathfrak a_\Theta$.
Define left ideals 
\begin{align*}
& J_{\Theta}(\lambda)
=\sum_{X\in \mathfrak{p}_\Theta}U(\mathfrak{g})(X-\lambda(X)), \\
& J_0(\lambda)
=\sum_{X\in \mathfrak{p}_0}U(\mathfrak{g})(X-\lambda(X)),\\
& J(\lambda)
=\sum_{X\in \mathfrak{b}}U(\mathfrak{g})(X-\lambda(X))
\end{align*}
of $U(\mathfrak{g})$. Then we can see easily that 
\[
J_{\Theta}(\lambda)=\{D\in U(\mathfrak{g})\,;\,
Df=0 \quad (\forall \,f\in\mathcal{F}(G/P_\Xi;\mathcal L_{\Xi,\mu}^\ell)\}.
\]

Let $a$ denote the anti-automorphism of $U(\mathfrak{g})$ defined 
by $a(X)=-X,\,a(XY)=YX$ for $X,\,Y\in\mathfrak{g}$. 

\begin{prop}
\label{prop:characterizedeg}
Assume that $I_{\Theta}(\lambda)$ is a 
two sided ideal of $U(\mathfrak{g})$ that satisfies 
\begin{equation}
\label{eqn:gap3}
J_{\Theta}(\lambda)=
I_{\Theta}(\lambda)+J_0(\lambda).
\end{equation}
Then 
\[
  \mathcal{F}(G/P_\Xi;\mathcal L_{\Xi,\mu}^\ell)=
\{
f\in \mathcal{F}(G/P;\mathcal L_{\mu + \rho(\Xi)}^\ell)\,;\,
\pi(a(D))f=0\quad (\forall\,D\in I_{\Theta}(\lambda))
\}
\]
\end{prop}
\pf Since 
\[
(\pi(a(D))f)(g)=((\text{Ad}(g^{-1})D)f)(g)\quad (\forall\,g\in G)
\]
and $I_{\Theta}(\lambda)$ is a two sided ideal 
of $U(\mathfrak{g})$, the proposition follows. 
\qed

\medskip
The above proposition shows that the two sided ideal 
$I_{\Theta}(\lambda)$, which satisfies (\ref{eqn:gap3})   
characterizes $\mathcal{F}(G/P_\Xi;\mathcal L_{\Xi,\mu}^\ell)$ in
$\mathcal{F}(G/P;\mathcal L_{\mu + \rho(\Xi)}^\ell)$
as a $U(\mathfrak{g})$-submodule. 
Notice that the condition
\begin{equation}
\label{eqn:gap}
J_{\Theta}(\lambda)=
I_{\Theta}(\lambda)+J(\lambda)
\tag{GAP}
\end{equation}
studied in \cite{O4,O5,O6,OO} implies (\ref{eqn:gap3}). 

\begin{thm}
Suppose that the Poisson transform 
\[
\mathcal{P}_{\mu+\rho(\Xi)}^\ell\,:\,
\mathcal{B}(G/P;\mathcal L_{\mu+\rho(\Xi)}^\ell)\rightarrow 
\mathcal{A}(G/K;\mathcal{M}_{\mu+\rho(\Xi)}^\ell)
\]
is bijective 
and assume that \eqref{eqn:gap3} holds for a two sided 
ideal $I_{\Theta}(\lambda)$ of $\mathfrak{g}$. 
Then the image of the Poisson transform of 
$\mathcal{B}(G/P_\Xi;\mathcal L_{\Xi,\mu}^\ell)$ is characterized by 
the system $\mathcal{M}_{\mu+\rho(\Xi)}^\ell$ 
together with the system defined by $I_{\Theta}(\lambda)$. 
\end{thm}
\pf
Since the Poisson transform and its inverse map (boundary value 
map) are both $G$-equivariant, 
Proposition~\ref{prop:characterizedeg} implies the theorem. 
\qed

\medskip
It is clear that there exists a two sided ideal $I_\Theta(\lambda)$
satisfying \eqref{eqn:gap} if and only if
\begin{equation}\label{eqn:AnnGap}
 J_\Theta(\lambda)=\Ann\bigl(M_\Theta(\lambda)\bigr)+J(\lambda).
\end{equation}
Here $M_\Theta(\lambda)$ is the generalized Verma module
$U(\mathfrak g)/J_\Theta(\lambda)$ and
\begin{align*}
\Ann\bigl(M_\Theta(\lambda)\bigr)
 &:=
  \{D\in U(\mathfrak g)\,;\,DM_\Theta(\lambda)=0\}\\
 &=\bigcap_{g\in G}\Ad(g)J_\Theta(\lambda)\\
 &=
  \{D\in U(\mathfrak g)\,;\,\pi(a(D))\mathcal{B}(G/P_\Xi;\mathcal L_{\Xi,\mu}^\ell)
  =0\}.
\end{align*}
The condition \eqref{eqn:AnnGap} is satisfied at least $\lambda$ is 
dominant and regular, namely,
\begin{equation}\label{eqn:DomReg}
 -\frac{\langle\lambda+\tilde\rho,\alpha\rangle}
  {\langle\alpha,\alpha\rangle}\notin
 \{0,1,2,3,\ldots\}\quad(\forall\alpha\in\Sigma(\mathfrak a)^+),
\end{equation}
which is a consequence of \cite[Theorem~3.12]{OO}.
Here $\tilde\rho$ is the half of the sum of the elements of $\Sigma(\alpha)^+$.
Hence \eqref{eqn:AnnGap} is satisfied for the harmonic case when
$\mu=\rho$, $\ell=0$ and $\lambda=0$.

When the complexification of $\mathfrak g$ equals $\mathfrak{gl}_N$,
Oshima \cite{O5} constructs
the generator system of $\Ann\bigl(M_\Theta(\lambda)\bigr)$ for any 
$\Theta\subset\Psi(\mathfrak a)$ and any character $\lambda$ of 
$\mathfrak p_\Theta$  through quantizations of elementary divisors
and gives necessary and sufficient condition for \eqref{eqn:AnnGap}
(cf.~\cite[Lemma 4.15]{OO}), which says that \eqref{eqn:AnnGap} is 
valid at least if $\lambda$ is regular, namely,
\begin{equation}\label{eqn:Reg}
 \frac{\langle\lambda+\tilde\rho,\alpha\rangle}
  {\langle\alpha,\alpha\rangle}\ne 0
  \quad(\forall\alpha\in\Sigma(\mathfrak a)^+).
\end{equation}
Oshima \cite{O6} constructs a generator system of a two sided ideal
$I_\Theta(\lambda)$ when $\mathfrak g$ is a real form of finite copies 
of classical complex Lie algebras $\mathfrak {gl}_n$, $\mathfrak {sp}_n$
or $\mathfrak o_{n}$ and shows that $I_\Theta(\lambda)$ satisfies 
\eqref{eqn:gap} if $\lambda$ is (strongly) regular, which will be 
explained in the next subsection.

\subsection{Minimal polynomials}
Oshima \cite{O6} constructed a set of generators of the annihilator of 
a generalized Verma module of the scalar type for  classical reductive 
Lie algebras. 
We review on the main result of \cite{O6} and discuss its implication for  
the Poisson transform on a degenerate series representation. 

Let $N$ be a positive integer and 
let $\mathfrak{gl}_N\simeq \text{End}(\mathbb{C}^N)$ be the general linear 
Lie algebra. Let $E_{ij}\in\ M(N,\mathbb{C})$ be the matrix whose $(i,j)$ entry is $1$ and 
the other entries are all zero. We have a triangular decomposition
\[
\mathfrak{gl}_N=\bar{\mathfrak{n}}_N+\mathfrak{a}_N+\mathfrak{n}_N,
\]
where 
\[
\mathfrak{a}_N=\sum_{j=1}^N \mathbb{C}E_{ii},
\quad 
\mathfrak{n}_N=\sum_{1\leq j<i\leq N}\mathbb{C}E_{ij},
\quad 
\bar{\mathfrak{n}}_N=\sum_{1\leq i<j\leq N}\mathbb{C}E_{ij}.
\]

Let $\mathfrak{g}$ be one of the classical complex Lie algebras $\mathfrak{gl}_n$, 
$\mathfrak{o}_{2n}$, $\mathfrak{o}_{2n+1}$ or 
$\mathfrak{sp}_n$ and put $N=n,\,2n,\,2n+1$, or $2n$, 
respectively, so that $\mathfrak{g}$ is a subalgebra of $\mathfrak{gl}_N$. 
Denoting \[\tilde I_n = \Bigl(\delta_{i, n+1-j}\Bigr)
   _{\substack{1\le i\le n\\ 1\le j\le n}}
   = \left(\begin{smallmatrix}  & & 1\\ & \udots & \\ 1 & & \end{smallmatrix}\right)
\text{ and }\tilde J_n = \left(\begin{smallmatrix}
   & \tilde I_n\\ -\tilde I_n
   \end{smallmatrix}\right),\]
we naturally identify
\begin{equation}\label{eq:inv}
\begin{aligned}
 \mathfrak o_n  &= \{X\in\mathfrak{gl}_n;\, \sigma_{\mathfrak o_{n}}(X) = X
 \}\quad \text{with }
 \sigma_{\mathfrak o_{n}}(X) = -\tilde I_n{}^tX\tilde I_n,\\
 \mathfrak{sp}_n &= \{X\in\mathfrak{gl}_{2n};\, \sigma_{\mathfrak{sp}_n}(X) = X
 \}\quad \text{with } \sigma_{\mathfrak{sp}_n}(X) = -\tilde J_n{}^tX\tilde J_n.
\end{aligned}
\end{equation}
Let $\sigma$ be the involutive automorphism of $\mathfrak{gl}_N$ defined as above 
so that $\mathfrak{g}=\mathfrak {gl}_N^\sigma
:=\{X\in\mathfrak{gl}_N\,;\,\sigma(X)=X\}$ 
(cf.~\cite[Definition~3.1]{O6}).
Put
$F_{ij} = E_{ij}$ if $\mathfrak g = \mathfrak{gl}_n$ and
$F_{ij} = E_{ij} + \sigma(E_{ij})$ with $\mathfrak g = \mathfrak{gl}_N^\sigma$ in 
other cases.  
Moreover putting $F_i=F_{ii}$ and 
$\mathbb F=(F_{ij})_{1\leq i,\,j\leq N}\in M(N,\mathfrak{g})$,
we have
\begin{equation}
 \Ad(g)q(\mathbb F)={}^tg\cdot q(\mathbb F)\cdot g^{-1}\quad(\forall g\in G)
\end{equation}
for any polynomial $q(x)$ and the analytic subgroup $G$ of $GL(n,\mathbb C)$
with the Lie algebra $\mathfrak g$.

We have a triangular decomposition of $\mathfrak{g}$
\[
\mathfrak{g}=\bar{\mathfrak{n}}+\mathfrak{a}+\mathfrak{n},
\]
where $\mathfrak{a}=\mathfrak{g}\cap\mathfrak{a}_N$, 
$\mathfrak{n}=\mathfrak{g}\cap\mathfrak{n}_N$ and 
$\bar{\mathfrak{n}}=\mathfrak{g}\cap\bar{\mathfrak{n}}_N$. 
Then $\mathfrak{a}$ is a Cartan subalgebra of $\mathfrak{g}$ and 
$\mathfrak{b}=\mathfrak{a}+\mathfrak{n}$ is a Borel subalgebra 
of $\mathfrak{g}$. 
Let $\Theta = \{0<n_1<n_2<\cdots<n_L=n\}$ be a sequence of strictly 
increasing positive integers ending at $n$ and put 
$H_\Theta=\sum_{k=1}^L\sum_{i=1}^{n_k}F_i$. 
Define 
\[
\begin{cases}
& \mathfrak{m}_\Theta=\{X\in\mathfrak{g}\,:\,\text{ad}(H_\Theta)X=0\}, \\
& \mathfrak{n}_\Theta=\{X\in\mathfrak{n}\,:\,\langle X,\mathfrak{m}_\Theta
\rangle=0\},
\quad 
\bar{\mathfrak{n}}_\Theta=\{X\in\bar{\mathfrak{n}}\,:\,\langle X,\mathfrak{m}_\Theta
\rangle=0\} \\
& \mathfrak{p}_\Theta=\mathfrak{m}_\Theta+\mathfrak{n}_\Theta.
\end{cases}
\]
Then $\mathfrak{p}_\Theta$ is a parabolic subalgebra of $\mathfrak{g}$ 
containing $\mathfrak{b}$. Put $H_{\bar{\Theta}}=\sum_{k=1}^{L-1}\sum_{i=1}^{n_k}
F_i$ and define $\mathfrak{m}_{\bar{\Theta}},\,\mathfrak{n}_{\bar{\Theta}},\,
\bar{\mathfrak{n}}_{\bar{\Theta}}$ and $\mathfrak{p}_{\bar{\Theta}}$ 
by replacing $\Theta$ by $\bar{\Theta}$ in the above definition.

For $1\leq i\leq n$ with $n_{j-1}<i\leq n_j$, 
put $\iota_\Theta(i)=j$.  
For $\lambda=(\lambda_1,\dots,\lambda_L)\in\mathbb{C}^L$ 
define a character of $\mathfrak{p}_\Theta$ by 
\[
\lambda\bigl(X+\sum_{i=1}^n C_iF_i\bigr)=\sum_{i=1}^n C_i\lambda_{\iota_\Theta(i)}
\quad\text{for $X\in 
\mathfrak{n}_\Theta+[\mathfrak{m}_\Theta,\mathfrak{m}_\Theta]$ 
and $C_i\in\mathbb{C}$}.
\]
In this subsection $U(\mathfrak g)$ denotes the universal enveloping algebra 
of the complex Lie algebra $\mathfrak g$.
The generalized Verma module 
$M_\Theta(\lambda)=U(\mathfrak g)/J_\Theta(\lambda)$ is
a quotient of the Verma module $M(\lambda)=U(\mathfrak g)/J(\lambda)$.
If $\lambda_L=0$, we similarly define a character of  
$\mathfrak{p}_{\bar{\Theta}}$, $J_{\bar{\Theta}}(\lambda)$ and 
$M_{\bar{\Theta}}(\lambda)$.

Define polynomials
\begin{equation}\label{eq:minipol}
\begin{cases}
 q_\Theta(\mathfrak{gl}_n;x,\lambda) = 
  \prod\limits_{j=1}^L(x - \lambda_j - n_{j-1}),\\
 q_\Theta(\mathfrak{o}_{2n+1};x,\lambda) = 
  (x-n)\prod\limits_{j=1}^L(x - \lambda_j - n_{j-1})
     (x + \lambda_j +  n_j - 2n),\\
 q_\Theta(\mathfrak{sp}_{n};x,\lambda) = 
   \prod\limits_{j=1}^L(x - \lambda_j - n_{j-1})
     (x + \lambda_j + n_j - 2n - 1),\\ 
 q_\Theta(\mathfrak{o}_{2n};x,\lambda) = 
   \prod\limits_{j=1}^L(x - \lambda_j - n_{j-1})
     (x + \lambda_j + n_j - 2n + 1)
\end{cases}
\end{equation}
and if $\mathfrak g=\mathfrak{sp}_n$, 
$\mathfrak{o}_{2n+1}$ or $\mathfrak{o}_{2n}$, 
\begin{equation*}
 q_{\bar\Theta}(\mathfrak g;x, \lambda) = 
  (x-n_{L-1})\prod_{j=1}^{L-1}(x-\lambda_j - n_{j-1})
      (x + \lambda_j + n_j -2n - \delta_{\mathfrak g})
\end{equation*}
with
\[
\delta_{\mathfrak g} =
\begin{cases}
 1\quad&\text{if }\mathfrak g = \mathfrak{sp}_n,\\
 0&\text{if }\mathfrak g = \mathfrak{so}_{2n+1}\text{ or }\mathfrak{gl}_n,\\
 -1&\text{if }\mathfrak g = \mathfrak{so}_{2n}.\\
\end{cases}
\]

Define two sided ideals of $U(\mathfrak g)$
\begin{equation}
\begin{cases}
 I_\Theta(\lambda) = \sum\limits_{i=1}^N\sum\limits_{j=1}^N
   U(\mathfrak g)
   q_\Theta(\mathfrak g; \mathbb F,\lambda)_{ij}
  + \sum\limits_{j\in J}U(\mathfrak g)
  \Bigl(\Delta_j - \lambda(\Delta_j)\Bigr),\\
 I_{\bar\Theta}(\lambda) = \sum\limits_{i=1}^N\sum\limits_{j=1}^N
   U(\mathfrak g)
   q_{\bar\Theta}(\mathfrak g; \mathbb F, \lambda)_{ij}
  + \sum\limits_{j\in \bar J}U(\mathfrak g)
  \Bigl(\Delta_j - \lambda(\Delta_j)\Bigr),
\end{cases}
\end{equation}
where $\Delta_1,\dots,\Delta_n$ are fixed generators of 
the center $Z(\mathfrak{g})\subset U(\mathfrak{g})$ 
with 
\[
\begin{cases}
 \ord \Delta_j=j\quad (1\leq j\leq n)
   &\text{if }\mathfrak g=\mathfrak{gl}_n,\\
 \ord \Delta_j=2j\quad (1\leq j\leq n)
   &\text{if }\mathfrak g=\mathfrak{o}_{2n+1}
\text{ or }\mathfrak g=\mathfrak{sp}_n,\\
 \ord \Delta_j=2j\quad (1\leq j< n),
\quad \ord  \Delta_n=n
   &\text{if }\mathfrak g=\mathfrak{o}_{2n}
\end{cases}
\]
and
\begin{equation}
\begin{cases}
 J = \{1,2,\ldots,L-1\}, \quad N = n
   &\text{if }\mathfrak g=\mathfrak{gl}_n,\\
 J = \{1,2,\ldots,L\},\ \bar J = \{1,2,\ldots,L-1\}, \quad N = 2n+1
   &\text{if }\mathfrak g=\mathfrak{o}_{2n+1},\\
 J = \bar J = \{1,2,\ldots,L-1\},\quad N = 2n
   &\text{if }\mathfrak g=\mathfrak{sp}_n,\\
 J = \bar J = \{1,2,\ldots,L-1\}\cup \{n\},\quad N = 2n
   &\text{if }\mathfrak g=\mathfrak{o}_{2n}.
\end{cases}
\end{equation}
Here $\lambda(\Delta_j)\in\mathbb C$ are defined so that 
$\Delta_j-\lambda(\Delta_j)\in\Ann\bigl(M(\lambda)\bigr)$.
When $\mathfrak g=\mathfrak o_{2n}$, $\Delta_j$ for $j=1,\dots,n-1$
are fixed and $\Delta_n$ is not fixed
by a non-trivial outer automorphism of $\mathfrak o_{2n}$.

Oshima~\cite{O6} studied sufficient conditions on $\lambda$ 
such that 
\begin{equation}
\label{eqn:gap0}
J_{\Theta'}(\lambda)=I_{\Theta'}(\lambda)+J(\lambda_{\Theta}) 
\end{equation}
with $\Theta'=\Theta$ or $\bar{\Theta}$. For $\mathfrak{g}=\mathfrak{gl}_n$, 
a necessary and sufficient condition on $\lambda$ for \eqref{eqn:gap0} 
is given (\cite[Remark 4.5 (i)]{O6}). 
In particular, in the case when $\mathfrak{g}=\mathfrak{gl}_n,\,
\mathfrak{o}_{2n+1}$ or $\mathfrak{sp}_n$, 
\eqref{eqn:gap0} holds 
for $\Theta'=\Theta$ if $\lambda|_\mathfrak{a}+\tilde\rho$ is regular, 
that is 
$\langle\lambda|_\mathfrak{a}+\tilde\rho,\alpha\rangle\not=0$ for 
any roots $\alpha\in\Sigma(\mathfrak{a})$.
When $\mathfrak{g}=\mathfrak{o}_{2n}$, \eqref{eqn:gap0} holds for $\Theta'=\Theta$ 
if $\lambda_\Theta|_\mathfrak{a}+\tilde\rho$ is strongly regular, 
that is  $\lambda|_\mathfrak{a}+\tilde\rho$ is not fixed by the
non-trivial outer automorphism of the root system $\Sigma(\mathfrak{a})$. 
Moreover, for $\mathfrak{g}=\mathfrak{gl}_n$, 
$\mathfrak{o}_{2n}$, $\mathfrak{o}_{2n+1}$ or 
$\mathfrak{sp}_n$, 
\eqref{eqn:gap0} holds for $\Theta=\bar{\Theta}$ if $\lambda$ satisfies 
the same regularity condition as above and $\lambda_L=0$. 
See Section~4 of \cite{O6} for details

The above construction of the minimal 
polynomial $q_\Theta(x,\lambda)$ can be extended for any complex 
reductive Lie algebra $\mathfrak{g}$ 
by considering a faithful representation $(\pi,\mathbb{C}^N)$ of $\mathfrak{g}$ 
(see \cite[Section 2]{O6}). 
Oshima and Oda~\cite{OO} studied sufficient conditions for the 
counterpart of (\ref{eqn:gap0}) with $\Theta'=\Theta$ for 
a general reductive Lie algebra $\mathfrak{g}$ 
(\cite[Theorem~3.21, Proposition~3.25, Proposition~3.27]{OO}). 
In particular, 
when $\mathfrak{g}$ is one of the simple exceptional 
Lie algebras $E_6,\,E_7,\,E_8,\,F_4$ or $G_2$, 
 the 
counterpart of (\ref{eqn:gap0}) with $\Theta'=\Theta$ associated with 
 non-trivial irreducible 
representation of $\mathfrak{g}$ with minimal degree holds if 
$\text{Re}\,\langle\lambda_\Theta+\tilde\rho,\alpha\rangle>0$ for all 
positive roots $\Sigma(\mathfrak{a})^+$ with 
respect to a Cartan subalgebra $\mathfrak{a}$ of $\mathfrak{g}$ 
(cf. \cite[Remark~4.13]{OO}). 

Lastly in this section
we list the order of the elements of $I_\Theta(\lambda)$ associated 
with the natural representation or the non-trivial representation with 
minimal degree according to the condition that $\mathfrak g$ is of 
classical type or exceptional type, respectively, when $\Theta$ corresponds
to a maximal parabolic subgroup $P_\Xi$ of $G$ as was described in the
previous subsection.
In the following Satake diagram the number attached to a simple root $\alpha
\in\Psi(\mathfrak a)$
indicates the order of the elements of $I_\Theta(\lambda)$ by the correspondence
$\Theta=\Psi(\mathfrak a_{\mathfrak p})\setminus
\{\alpha|_{\mathfrak a_{\mathfrak p}}\}$.
The order is easily seen from \eqref{eq:minipol} if $\mathfrak g$ is of 
classical type and it is given in \cite[\S4]{OO} if $\mathfrak g$ is of 
exceptional type.

The dotted circles correspond to the Shilov boundaries 
in Hermitian cases.

When $\mathfrak g$ is a complex simple Lie algebra, 
the degree is obtained by the corresponding Dynkin 
diagram with no arrow and no black circle.
\medskip

\centerline
{``Degree of minimal polynomials associated to natural/smallest representations"}

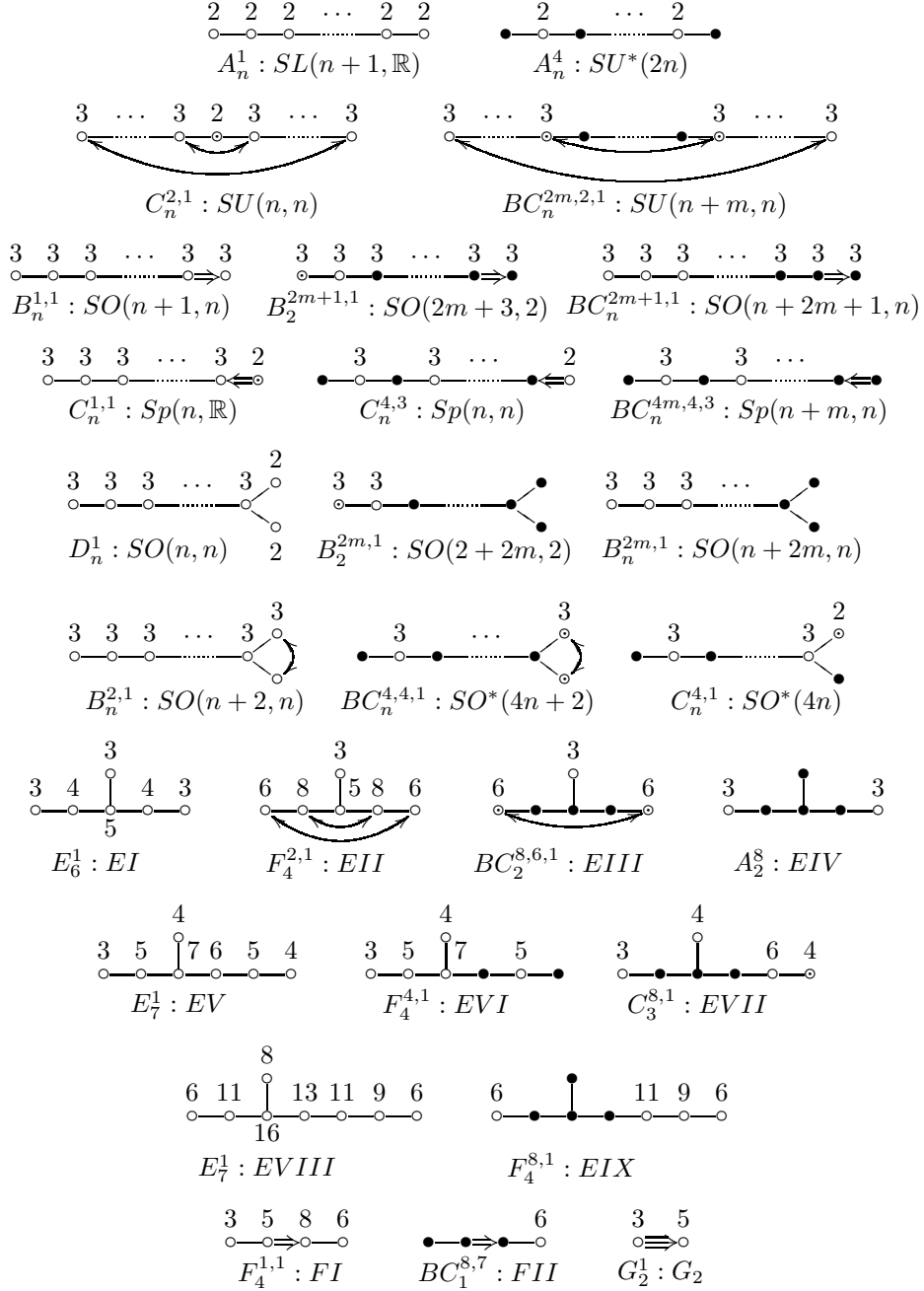
\vspace*{-5mm}
\begin{gather*}
\begin{xy}
 (14,-4) *{A_n^1:SL(n+1,\mathbb R)};
 \ar@{-} (0,0); *++!D{2} *{\circ} ;
  (5,0)  *++!D{2} *{\circ}="B"
 \ar@{-} "B";(10,0) *++!D{2} *{\circ}="C"
 \ar@{-} "C";(14,0) \ar@{.} (14,0);(19,0)^*!U{\cdots}
 \ar@{-} (19,0);(23,0) *++!D{2} *{\circ}="H"
 \ar@{-} "H";(28,0): *++!D{2} *{\circ}
\end{xy}
\qquad
\begin{xy}
 (14,-4) *{A_n^4:SU^*(2n)};
 \ar@{-} *{\bullet} ; (5,0)  *++!D{2} *{\circ}="B"
 \ar@{-} "B";(10,0) *{\bullet}="C"
 \ar@{-} "C";(14,0) \ar@{.} (14,0);(19,0)^*!U{\cdots}
 \ar@{-} (19,0);(23,0) *++!D{2} *{\circ}="H"
 \ar@{-} "H";(28,0) *{\bullet}
\end{xy}
\allowdisplaybreaks\\  
\begin{xy}
 (20,-9) *{C_n^{2,1}:SU(n,n)};
 \ar@{-}  *++!D{3} *{\circ}="A";
 \ar@{-} "A";(4,0) \ar@{.} (4,0);(9,0)^*!U{\cdots}
 \ar@{-} (9,0);(13,0) *++!D{3} *{\circ}="B"
 \ar@{-} "B";(18,0) *++!D{2} *{\cdot} *{\circ}="C"
 \ar@{-} "C";(23,0) *++!D{3} *{\circ}="D"
 \ar@{-} "D";(27,0) \ar@{.} (27,0);(32,0)^*!U{\cdots}
 \ar@{-} (32,0);(36,0) *++!D{3} *{\circ}="E"
 \ar@/_5mm/ @{<->} "A";"E"
 \ar@/_/ @{<->} "B";"D"
\end{xy}
\qquad
\begin{xy}
  (26,-9) *{BC_n^{2m,2,1}:SU(n+m,n)};
 \ar@{-}  *++!D{3} *{\circ}="A";
 \ar@{-} "A";(4,0) \ar@{.} (4,0);(9,0)^*!U{\cdots}
 \ar@{-} (9,0);(13,0) *++!D{3} *{\cdot} *{\circ}="B"
 \ar@{-} "B";(18,0) *{\bullet}="C"
 \ar@{-} "C";(22,0) \ar@{.} (22,0);(27,0)
 \ar@{-} (27,0);(31,0) *{\bullet}="D"
 \ar@{-} "D";(36,0) *++!D{3} *{\cdot} *{\circ}="E"
 \ar@{-} "E";(41,0) \ar@{.} (41,0);(46,0)^*!U{\!\!\!\cdots}
 \ar@{-} (46,0);(51,0) *++!D{3} *{\circ}="F"
 \ar@/_6mm/ @{<->} "A";"F"
 \ar@/_2mm/ @{<->} "B";"E"
\end{xy}\allowdisplaybreaks\\ 
\begin{xy}
 (14,-4) *{B_n^{1,1}:SO(n+1,n)};
 \ar@{-}  *++!D{3} *{\circ} ;
  (5,0)  *++!D{3} *{\circ}="B"
 \ar@{-} "B";(10,0) *++!D{3} *{\circ}="C"
 \ar@{-} "C";(14,0) \ar@{.} (14,0);(19,0)^*!U{\cdots}
 \ar@{-} (19,0);(23,0) *++!D{3} *{\circ}="H"
 \ar@{=>} "H";(28,0) *++!D{3} *{\circ}
\end{xy}
\ \ 
\begin{xy}
 (14,-4) *{B_2^{2m+1,1}:SO(2m+3,2)};
 \ar@{-}  *++!D{3} *{\cdot} *{\circ} ;
  (5,0)  *++!D{3}  *{\circ}="B"
 \ar@{-} "B";(10,0) *++!D{3} *{\bullet}="C"
 \ar@{-} "C";(14,0) \ar@{.} (14,0);(19,0)^*!U{\cdots}
 \ar@{-} (19,0);(23,0) *++!D{3} *{\bullet}="H"
 \ar@{=>} "H";(28,0) *++!D{3} *{\bullet}
\end{xy}
\ \ 
\begin{xy}
 (18,-4) *{BC_n^{2m+1,1}:SO(n+2m+1,n)};
 \ar@{-}  *++!D{3} *{\circ} ;
  (5,0)  *++!D{3} *{\circ}="B"
 \ar@{-} "B";(10,0) *++!D{3} *{\circ}="C"
 \ar@{-} "C";(14,0) \ar@{.} (14,0);(19,0)^*!U{\cdots}
 \ar@{-} (19,0);(23,0) *++!D{3} *{\bullet}="H"
 \ar@{-} "H";(28,0) *++!D{3} *{\bullet}="I"
 \ar@{=>} "I";(33,0) *++!D{3} *{\bullet}
\end{xy}
\allowdisplaybreaks\\ 
\begin{xy}
 (14,-4) *{C_n^{1,1}:Sp(n,\mathbb R)};
 \ar@{-}  *++!D{3} *{\circ} ;
  (5,0)  *++!D{3} *{\circ}="B"
 \ar@{-} "B";(10,0) *++!D{3} *{\circ}="C"
 \ar@{-} "C";(14,0) \ar@{.} (14,0);(19,0)^*!U{\cdots}
 \ar@{-} (19,0);(23,0) *++!D{3} *{\circ}="H"
 \ar@{<=} "H";(28,0) *++!D{2} *{\cdot} *{\circ}
\end{xy}
\ \ \ \ 
\begin{xy}
 (16,-4) *{C_n^{4,3}:Sp(n,n)};
 \ar@{-} *{\bullet} ;
  (5,0)  *++!D{3} *{\circ}="B"
 \ar@{-} "B";(10,0) *{\bullet}="C"
 \ar@{-} "C";(15,0) *++!D{3} *{\circ}="D"
 \ar@{-} "D";(19,0) \ar@{.} (19,0);(24,0)^*!U{\cdots}
 \ar@{-} (24,0);(28,0) *{\bullet}="H"
 \ar@{<=} "H";(33,0) *++!D{2} *{\circ}
\end{xy}
\ \ 
\begin{xy}
 (16,-4) *{BC_n^{4m,4,3}:Sp(n+m,n)};
 \ar@{-}  *{\bullet} ;
  (5,0)  *++!D{3} *{\circ}="B"
 \ar@{-} "B";(10,0) *{\bullet}="C"
 \ar@{-} "C";(15,0) *++!D{3} *{\circ}="D"
 \ar@{-} "D";(19,0) \ar@{.} (19,0);(24,0)^*!U{\cdots}
 \ar@{-} (24,0);(28,0) *{\bullet}="H"
 \ar@{<=} "H";(33,0) *{\bullet}
\end{xy}
\allowdisplaybreaks\\
\begin{xy}
 (10,-6) *{D_n^{1}:SO(n,n)};
 \ar@{-}  *++!D{3} *{\circ} ;
  (5,0)  *++!D{3} *{\circ}="B"
 \ar@{-} "B";(10,0) *++!D{3} *{\circ}="C"
 \ar@{-} "C";(14,0) \ar@{.} (14,0);(19,0)^*!U{\cdots}
 \ar@{-} (19,0);(23,0) *++!D{3} *{\circ}="H"
 \ar@{-} "H";(27,3) *++!D{2} *{\circ}
 \ar@{-} "H";(27,-3) *++!U{2} *{\circ}
\end{xy}
\ \ 
\begin{xy}
 (14,-6) *{B_2^{2m,1}:SO(2+2m,2)};
 \ar@{-}  *++!D{3} *{\cdot} *{\circ} ;
  (5,0)  *++!D{3} *{\circ}="B"
 \ar@{-} "B";(10,0) *{\bullet}="C"
 \ar@{-} "C";(14,0) \ar@{.} (14,0);(19,0) 
 \ar@{-} (19,0);(23,0) *{\bullet}="H"
 \ar@{-} "H";(27,3) *{\bullet}
 \ar@{-} "H";(27,-3) *{\bullet}
\end{xy}
\ \ 
\begin{xy}
 (16,-6) *{B_n^{2m,1}:SO(n+2m,n)};
 \ar@{-}  *++!D{3} *{\circ} ;
  (5,0)  *++!D{3} *{\circ}="B"
 \ar@{-} "B";(10,0) *++!D{3} *{\circ}="C"
 \ar@{-} "C";(14,0) \ar@{.} (14,0);(19,0)^*!U{\cdots}
 \ar@{-} (19,0);(23,0) *{\bullet}="H"
 \ar@{-} "H";(27,3) *{\bullet}
 \ar@{-} "H";(27,-3) *{\bullet}
\end{xy}
\allowdisplaybreaks\\ 
\begin{xy}
 (16,-6) *{B_n^{2,1}:SO(n+2,n)};
 \ar@{-}  *++!D{3} *{\circ} ;
  (5,0)  *++!D{3} *{\circ}="B"
 \ar@{-} "B";(10,0) *++!D{3} *{\circ}="C"
 \ar@{-} "C";(14,0) \ar@{.} (14,0);(19,0)^*!U{\cdots}
 \ar@{-} (19,0);(23,0) *++!D{3}  *{\circ}="H"
 \ar@{-} "H";(27,3)  *++!D{3} *{\circ}="I"
 \ar@{-} "H";(27,-3)  *{\circ}="J"
 \ar@/^/ @{<->} "I";"J"
\end{xy}
\ \ \ \ 
\begin{xy}
 (14,-6) *{BC_n^{4,4,1}:SO^*(4n+2)};
 \ar@{-} *{\bullet} ;
  (5,0)  *++!D{3} *{\circ}="B"
 \ar@{-} "B";(10,0) *{\bullet}="C"
 \ar@{-} "C";(14,0) \ar@{.} (14,0);(19,0)^*!U{\cdots}
 \ar@{-} (19,0);(23,0) *{\bullet}="H"
 \ar@{-} "H";(27,3)  *++!D{3} *{\cdot} *{\circ}="I"
 \ar@{-} "H";(27,-3)  *{\cdot} *{\circ}="J"
 \ar@/^/ @{<->} "I";"J"
\end{xy}
\ \ \ \ 
\begin{xy}
 (16,-6) *{C_n^{4,1}:SO^*(4n)};
 \ar@{-} *{\bullet} ;
  (5,0)  *++!D{3} *{\circ}="B"
 \ar@{-} "B";(10,0) *{\bullet}="C"
 \ar@{-} "C";(14,0) \ar@{.} (14,0);(19,0) 
 \ar@{-} (19,0);(23,0) *++!D{3} *{\circ}="H"
 \ar@{-} "H";(27,3) *++!D{2}  *{\cdot} *{\circ}
 \ar@{-} "H";(27,-3) *{\bullet}
\end{xy}
\allowdisplaybreaks\\ 
\begin{xy}
 (8,-7) *{E_6^1:EI};
 \ar@{-} (0,0); *++!D{3} *{\circ} ;
  (5,0)  *++!D{4} *{\circ}="B"
 \ar@{-} "B";(10,0) *+!U{5} *{\circ}="C"
 \ar@{-} "C";(10,5) *++!D{3} *{\circ}
 \ar@{-} "C";(15,0) *++!D{4} *{\circ}="D"
 \ar@{-} "D";(20,0) *++!D{3} *{\circ}
\end{xy}
\ \ \ \ 
\begin{xy}
 (8,-7) *{F_4^{2,1}:EII};
 \ar@{-} (0,0); *++!D{6} *{\circ}="A" ;
  (5,0)  *++!D{8} *{\circ}="B"
 \ar@{-} "B";(10,0) *+!L+!D{5} *{\circ}="C"
 \ar@{-} "C";(10,5) *++!D{3} *{\circ}
 \ar@{-} "C";(15,0) *++!D{8} *{\circ}="D"
 \ar@{-} "D";(20,0) *++!D{6} *{\circ}="E"
 \ar@/_4mm/ @{<->} "A";"E"
 \ar@/_/ @{<->} "B";"D"
\end{xy}
\ \ \ \ 
\begin{xy}
 (8,-7) *{BC_2^{8,6,1}:EIII};
 \ar@{-} (0,0); *++!D{6} *{\cdot} *{\circ}="A" ;
  (5,0)  *{\bullet}="B"
 \ar@{-} "B";(10,0) *{\bullet}="C"
 \ar@{-} "C";(10,5) *++!D{3} *{\circ}
 \ar@{-} "C";(15,0) *{\bullet}="D"
 \ar@{-} "D";(20,0) *++!D{6} *{\cdot} *{\circ}="E"
 \ar@/_/ @{<->} "A";"E"
\end{xy}
\ \ \ \ 
\begin{xy}
 (8,-7) *{A_2^8:EIV};
 \ar@{-} (0,0); *++!D{3} *{\circ}="A" ;
  (5,0)  *{\bullet}="B"
 \ar@{-} "B";(10,0) *{\bullet}="C"
 \ar@{-} "C";(10,5)  *{\bullet}
 \ar@{-} "C";(15,0)  *{\bullet}="D"
 \ar@{-} "D";(20,0) *++!D{3} *{\circ}="E"
\end{xy}
\allowdisplaybreaks\\ 
\begin{xy}
 (10,-4) *{E_7^1:EV};
 \ar@{-} (0,0); *++!D{3} *{\circ} ;
  (5,0)  *++!D{5} *{\circ}="B"
 \ar@{-} "B";(10,0) *+!L+!D{7} *{\circ}="C"
 \ar@{-} "C";(10,5) *++!D{4} *{\circ}
 \ar@{-} "C";(15,0) *++!D{6} *{\circ}="D"
 \ar@{-} "D";(20,0) *++!D{5} *{\circ}="E"
 \ar@{-} "E";(25,0) *++!D{4} *{\circ}
\end{xy}
\ \ \ \ 
\begin{xy}
 (10,-4) *{F_4^{4,1}:EVI};
 \ar@{-} (0,0); *++!D{3} *{\circ} ;
  (5,0)  *++!D{5} *{\circ}="B"
 \ar@{-} "B";(10,0) *+!L+!D{7} *{\circ}="C"
 \ar@{-} "C";(10,5) *++!D{4} *{\circ}
 \ar@{-} "C";(15,0) *{\bullet}="D"
 \ar@{-} "D";(20,0) *++!D{5} *{\circ}="E"
 \ar@{-} "E";(25,0) *{\bullet}
\end{xy}
\ \ \ \ 
\begin{xy}
 (10,-4) *{C_3^{8,1}:EVII};
 \ar@{-} (0,0); *++!D{3} *{\circ}="A" ;
  (5,0)  *{\bullet}="B"
 \ar@{-} "B";(10,0) *{\bullet}="C"
 \ar@{-} "C";(10,5) *++!D{4} *{\circ}
 \ar@{-} "C";(15,0) *{\bullet}="D"
 \ar@{-} "D";(20,0) *++!D{6} *{\circ}="E"
 \ar@{-} "E";(25,0) *++!D{4} *{\cdot} *{\circ}
\end{xy}
\allowdisplaybreaks\\ 
\begin{xy}
 (10,-7) *{E_7^1:EVIII};
 \ar@{-} (0,0); *++!D{6} *{\circ} ;
  (5,0)  *++!D{11} *{\circ}="B"
 \ar@{-} "B";(10,0) *+!U{16} *{\circ}="C"
 \ar@{-} "C";(10,5) *++!D{8} *{\circ}
 \ar@{-} "C";(15,0) *++!D{13} *{\circ}="D"
 \ar@{-} "D";(20,0) *++!D{11} *{\circ}="E"
 \ar@{-} "E";(25,0) *++!D{9} *{\circ}="F"
 \ar@{-} "F";(30,0) *++!D{6} *{\circ}
\end{xy}
\ \ \ \ 
\begin{xy}
 (10,-7) *{F_4^{8,1}:EIX};
 \ar@{-} (0,0); *++!D{6} *{\circ}="A" ;
  (5,0)  *{\bullet}="B"
 \ar@{-} "B";(10,0) *{\bullet}="C"
 \ar@{-} "C";(10,5) *{\bullet}
 \ar@{-} "C";(15,0) *{\bullet}="D"
 \ar@{-} "D";(20,0) *++!D{11} *{\circ}="E"
 \ar@{-} "E";(25,0) *++!D{9} *{\circ}="F"
 \ar@{-} "F";(30,0) *++!D{6} *{\circ}
\end{xy}
\allowdisplaybreaks\\ 
\begin{xy}
 (8,-4) *{F_4^{1,1}:FI};
 \ar@{-}  *++!D{3} *{\circ} ;
  (5,0)  *++!D{5} *{\circ}="B"
 \ar@{=>} "B";(10,0) *++!D{8} *{\circ}="C"
 \ar@{-} "C";(15,0) *++!D{6} *{\circ}="D"
\end{xy}
\qquad
\begin{xy}
 (8,-4) *{BC_1^{8,7}:FII};
 \ar@{-}  *{\bullet} ;
  (5,0)   *{\bullet}="B"
 \ar@{=>} "B";(10,0) *{\bullet}="C"
 \ar@{-} "C";(15,0) *++!D{6} *{\circ}="D"
\end{xy}
\qquad
\begin{xy}
 (3,-4) *{G_2^{1}:G_2};
 \ar@3{->}  *++!D{3} *{\circ} ;
  (6,0)   *++!D{5} *{\circ}
\end{xy}
\end{gather*}
\begin{rem}
The restricted root system is shown by the notation in \cite[Appendix]{OS1} such 
as $BC_{n}^{m_1,m_2,m_3}$ and the Lie algebra $\mathfrak m_\Xi$ and its 
complexification for any $\Xi\subset \Psi(\mathfrak a_{\mathfrak p})$ 
can be easily read from the Satake diagram as was explained in 
\cite[Appendix B]{OS2}.
Namely, if $G$ is semisimple, the subdiagram corresponding to 
$\Psi(\Theta)=\{\alpha\in \Psi(\alpha)\,;\,\alpha|_{\mathfrak a_{\mathfrak p}}\in 
\Theta\cup \{0\}\}$ is the Satake diagram of $\mathfrak m_\Xi$.

If $G$ is a connected real form of a simply connected semisimple complex 
Lie group, $M_{\Xi,s}^o$ is a real form of a simply connected complex Lie 
group and $M_\Xi/M_{\Xi}^o$ is isomorphic to the direct sum of 
$(\mathbb Z/2\mathbb Z)^k$ and $\ell$ copies of $U(1)$.
Here $M_{\Xi}^o$ is the identity component of $M_\Xi$,
$\ell$ is the number of the arrows pointing
roots in $\Psi(\mathfrak a)\setminus\Psi(\Theta)$ and
$k$ is the number of the roots in $\Psi(\mathfrak a)\setminus\Psi(\Theta)$ which
are not pointed by any arrow and are not directly linked  by any line
attached to any root $\alpha\in\Psi(\mathfrak a)$ with
$\alpha|_{\mathfrak a_{\mathfrak p}}=0$.
\end{rem}


\section{Examples}
In this section we examine in detail 
the differential operators on a homogeneous
line bundle over a Riemannian symmetric space $G/K$ induced
from the two sided ideal given in the preceding section for 
$G=U(p,q),\,Sp(n,\mathbb{R})$ and $GL(n,\mathbb{R})$. 
\subsection{$U(p,q)$}
Let $\sigma$ be the complex linear involution of 
$\mathfrak g = \mathfrak{gl}_{p+q}$ defined by
\[\sigma(X) = I_{p,q}XI_{p,q}\quad\text{ \ with \ } 
  I_{p,q}:=
  \begin{pmatrix}
  I_p & 0\\ 0 & -I_q
  \end{pmatrix}.
\]
Here $1\le q\le p$.
Then $G=U(p,q)
=\{g\in GL(p+q,\mathbb C)\,;\,g=I_{p,q}\overline{{}^tg^{-1}}I_{p,q}\}$
and $K=U(p,q)\cap U(p+q)=U(p)\times U(q)$.

The corresponding Satake diagram and the Dynkin diagram of the restricted root
system are as follows.
\begin{align*}
 &U(p,q)\quad(p=q):
 \begin{xy}
 \ar@{-}  *++!D{\tilde \alpha_1} *{\circ}="A";
 \ar@{-} "A";(4,0) \ar@{.} (4,0);(9,0)^*!U{\!\!\!\cdots}
 \ar@{-} (9,0);(13,0) *++!D{\!\!\tilde\alpha_{q-1}} *{\circ}="B"
 \ar@{-} "B";(18,0) *++!D{\tilde\alpha_q} *{\cdot} *{\circ}="C"
 \ar@{-} "C";(23,0) *++!D{\;\,\tilde\alpha_{q+1}} *{\circ}="D"
 \ar@{-} "D";(27,0) \ar@{.} (27,0);(32,0)^*!U{\!\cdots}
 \ar@{-} (32,0);(36,0) *++!D{\tilde\alpha_{2q-1}} *{\circ}="E"
 \ar@/_5mm/ @{<->} "A";"E"
 \ar@/_/ @{<->} "B";"D"
\end{xy}
&\Longrightarrow
\begin{xy}
 \ar@{-}  *++!D{\alpha_1} *{\circ}="A";
 \ar@{-} "A";(4,0) \ar@{.} (4,0);(9,0)^*!U{\!\!\!\cdots}
 \ar@{-} (9,0);(13,0) *++!D{\!\!\alpha_{q-1}} *{\circ}="B"
 \ar@{<=} "B";(18,0) *++!D{\;\alpha_q} *{\cdot} *{\circ}="C"
\end{xy}\\
 &U(p,q)\quad(p>q):
\begin{xy}
 \ar@{-}  *++!D{\tilde\alpha_1} *{\circ}="A";
 \ar@{-} "A";(4,0) \ar@{.} (4,0);(9,0)^*!U{\!\!\!\cdots}
 \ar@{-} (9,0);(13,0) *++!D{\tilde\alpha_{q-1}} *{\cdot} *{\circ}="B"
 \ar@{-} "B";(18,0) *{\bullet}="C"
 \ar@{-} "C";(22,0) \ar@{.} (22,0);(27,0)^*!U{\!\!\!\cdots}
 \ar@{-} (27,0);(31,0) *{\bullet}="D"
 \ar@{-} "D";(36,0) *++!D{\tilde\alpha_{p-1}} *{\cdot} *{\circ}="E"
 \ar@{-} "E";(41,0) \ar@{.} (41,0);(46,0)^*!U{\!\!\!\cdots}
 \ar@{-} (46,0);(51,0) *++!D{\tilde\alpha_{p+q-1}} *{\circ}="F"
 \ar@/_6mm/ @{<->} "A";"F"
 \ar@/_2mm/ @{<->} "B";"E"
\end{xy}
&\Longrightarrow
\begin{xy}
 \ar@{-}  *++!D{\alpha_1} *{\circ}="A";
 \ar@{-} "A";(4,0) \ar@{.} (4,0);(9,0)^*!U{\!\!\!\cdots}
 \ar@{-} (9,0);(13,0) *++!D{\!\!\alpha_{q-1}} *{\circ}="B"
 \ar@{=>} "B";(18,0) *++!D{\;\alpha_q} *{\cdot} *{\circ}="C"
\end{xy}\\
\end{align*}
In this subsection, we restrict ourselves to a parabolic 
subgroup $P_\Xi$ that 
satisfies (\ref{eq:4.48}) for $\ell\not=0$. This is the case 
$\Xi\subset \Psi(\mathfrak{a}_\mathfrak{p})\setminus \{\alpha_q\}$. 
We fix $L+1$ non-negative integers $0=n_0<n_1<\cdots<n_L=q$ and 
put $\Xi=\{\alpha_i\,;\,i\in\{1,\dots,q\}\setminus\{n_1,n_2,\dots,n_L\}\}$.
Then 
\[
 \Theta=\{\tilde\alpha_\nu\,;\,q-1<\nu<p-1\}
 \cup\bigcup_{j=1}^L\{\tilde\alpha_\nu,\tilde\alpha_{p+q-\nu}\,;\,
   n_{j-1}<\nu<n_j\}.
\]
We have 
\begin{equation}
 \mathfrak m_{\Xi}+\mathfrak a_\Xi\simeq
 \mathfrak{gl}_{n_1}\oplus\mathfrak{gl_{n_2-n_1}}\oplus\cdots
  \oplus\mathfrak{gl}_{n_L-n_{L-1}}\oplus\mathfrak u(p-q).
\end{equation}
In particular, if $L=1$, then $\Xi=\Psi(\mathfrak{a}_\mathfrak{p})\setminus 
\{\alpha_q\}$ and 
 $G/P_\Xi$ is the Shilov boundary of $G/K$.  

We examine the system of differential equations characterizing the image of the
Poisson transform $\mathcal P_{\Xi,\mu}^\ell$ of the space of hyperfunction 
sections over the boundary $G/P_\Xi$ of $G/K$. 
If $p>q$, then the differential operators corresponding to the 
generators of the two sided ideal of $U(\mathfrak{g})$ given 
by the  minimal polynomial  described 
in Section~3 is of  order $2L+1$, but we will show 
that the image of the Poisson transform 
can be characterized by operators of order at most $2L$, 
by reducing the operators in the two sided ideals modulo 
$\sum_{X\in\mathfrak{k}}U(\mathfrak{g})(X+\chi_\ell(X))$ and 
taking a $K$-invariant left ideal (cf.\ Theorem~\ref{thm:upq}, 
Corollary~\ref{cor:upq}). 
For $L=1$ these second order differential operators 
are the Hua operators (cf.\ Remark~\ref{rem:upq}). 

Let $i, j, k, \ell, \mu$ and $\nu$ are indices which satisfy
\[
 1\le i,\ j\le q\quad\text{and}\quad q<k,\ \ell\le p
\quad\text{and}\quad 1\le\mu,\ \nu\le p
\]
and put
\[
 \bar i=p+q+1-i,\quad \bar j = p+q+1-j.
\] 

Put
\[
 \mathfrak k_{\mathbb C} = \sum_{\mu, \nu}\mathbb CE_{\mu,\nu} + \sum_{i,j}\mathbb CE_{\bar i,\bar j}
\]
and
\[
 \mathfrak a_{\mathbb C} = \sum_{i=1}^q\mathbb CE_i
\text{ \ with \ }
 E_i = E_{i,\bar i} + E_{\bar i,i}.
\]
Let $e_i\in\mathfrak a_{\mathbb C}^*$ defined by
\[
  e_i(E_j) = \delta_{i,j}.
\]
Define
\begin{align*}
 Y_i       &= -E_{i,i} + E_{i,\bar i}  - E_{\bar i,i} + E_{\bar i,\bar i},\\
 Y_{i,k}   &= E_{i,k} + E_{\bar i,k},\qquad
 Y_{k,i}   = E_{k,i} - E_{k, \bar i},\\
 Y_{i,j,+} &= E_{i,j} + E_{\bar i,j} - E_{i,\bar j} - E_{\bar i,\bar j}
 \quad\text{for }i\ne j,\\
 Y_{i,j,1} &= E_{i,j} + E_{\bar i,j} + E_{i,\bar j} + E_{\bar i,\bar j}
 \quad\text{for }i<j,\\
 Y_{i,j,2} &= E_{j,i} - E_{\bar j,i} - E_{j,\bar i} + E_{\bar j,\bar i}
 \quad\text{for }i<j
\end{align*}
and let $\mathfrak n_\mathbb C$ be the nilpotent subalgebra of 
$\mathfrak{gl}_{p+q}$
spanned by $Y_i$, $Y_{i,k}$, $Y_{i,j,+}$ with $i\ne j$, $Y_{i,j,1}$ and
$Y_{i,j,2}$ with $i<j$.
Then $\mathfrak{gl}_{p+q}=\mathfrak g_{\mathbb C}=
\mathfrak k_{\mathbb C}+\mathfrak a_{\mathbb C}+\mathfrak n_{\mathbb C}$
is the complexification of the Iwasawa decomposition 
$\mathfrak u(p,q)=\mathfrak k+\mathfrak a_{\mathfrak p}+\mathfrak n$ of 
the Lie algebra of $U(p,q)$.

For a polynomial $f(x)$ we will examine $\mathbb F\in 
M\bigl(p+q,U(\mathfrak g)\bigr)$ 
defined by $\mathbb F=f(\mathbb E)$ with 
$\mathbb E=(E_{i,j})\in M\bigl(p+q,U(\mathfrak g)\bigr)$.
Note that $V_f=\sum_{i,j}\mathbb C F_{i,j}$ is 
a $\mathfrak g$-module by the adjoint action of $\mathfrak g$ and 
it is decomposed into 4 $\mathfrak k$-submodules.
\begin{equation}
  V_{f}
 =\bigoplus_{\epsilon_1,\epsilon_2=0,1} V_{f}^{\epsilon_1,\epsilon_2}
\text{ \ with \ }
 V_{f}^{\epsilon_1,\epsilon_2}:=
 \sum_{\substack{p\epsilon_1<i\le p+q\epsilon_1\\
       p\epsilon_2<j\le p+q\epsilon_2}}
 \mathbb C F_{i,j}.
\end{equation}
We will calculate $\gamma_\ell(F_{i,j})$ (cf.~\eqref{eq:4.36})
for $\mathbb F=(F_{i,j})$ to get $V_{f}^{\epsilon_1,\epsilon_2}$
killing the image of the Poisson transform $\mathcal P_{\Xi,\mu}^\ell$.
A similar calculation was done in the proof of \cite[Proposition~3.4]{O6}.
The polynomial $f(x)$ so that $V_f$ characterizes the image of
$\mathcal P_{\Xi,\mu}^\ell$ is given in the preceding section and then
the degree of $f(x)$ which is the maximal order of the elements of $V_f$ 
equals $2L+1$ or $2L$ if $p>q$ or $p=q$, respectively.
It happens that $V_{f}$ doesn't kill the image but 
$V_f^{\epsilon_1,\epsilon_2}$ does so for suitable $f(x)$ and 
$(\epsilon_1,\epsilon_2)$ and then we will get the system of 
differential equations of order $\le 2L$ characterizing the image also 
in the case $p>q$.

Note that for $H\in \mathfrak a_{\mathbb C}$
\begin{align*}
 [H, Y_i] &= 2e_i(H)Y_i,\\
 [H, Y_{i,k}] &= e_i(H)Y_{i,k},&
 [H, Y_{k,i}] &= e_i(H)Y_{k,i},\\
 [H, Y_{i,j,+}] &= (e_i+e_j)(H)Y_{i,j,+},\\
 [H, Y_{i,j,1}] &= (e_i-e_j)(H)Y_{i,j,1},&
 [H, Y_{i,j,2}] &= (e_i-e_j)(H)Y_{i,j,2}.
\end{align*}
Then the root system $\Sigma(\mathfrak a_{\mathfrak p})$ is 
of type $BC_q^{2(p-q),2,1}$ and
\begin{align*}
\Psi(\mathfrak a_{\mathfrak p})&=\{e_1-e_2,e_2-e_3,\dots,e_{q-1}-e_q,e_q\},
\allowdisplaybreaks\\
\rho&=(p+q-1)e_1+(p+q-2)e_2+\cdots+(p-q+1)e_q,\\
 E_{i,\bar i} &= \frac12 E_i + \frac12 Y_i + \frac12(E_{i,i} - E_{\bar i,\bar i}),\\
 E_{\bar i,i} &= \frac12 E_i - \frac12 Y_i - \frac12(E_{i,i} - E_{\bar i,\bar i}),\\
 E_{k,\bar i} &= -Y_{k,i} + E_{k,i},\qquad
 E_{\bar i,k} = Y_{i,k} - E_{i,k},\\
 E_{i,\bar j} &= \frac12(Y_{i,j,1} - Y_{i,j,+}) - E_{\bar i,\bar j}
 \quad\text{for }i<j,\\
 E_{\bar i, j} &= \frac12(Y_{i,j,1} + Y_{i,j,+}) - E_{i,j}
 \quad\text{for }i<j,\\
 E_{i,\bar j} &= -\frac12(Y_{i,j,+} + Y_{j,i,2}) + E_{i,j}
 \quad\text{for }i>j,\\
 E_{\bar i,j} &= \frac12(Y_{i,j,+} - Y_{j,i,2}) + E_{\bar i,\bar j}
 \quad\text{for }i>j.
\end{align*}

Suppose $F_{a,b}\in U(\mathfrak g)$ for $1\le a,b\le p+q$ satisfy
\[
 [E_{i,j}, F_{a,b}] = \delta_{j,a}F_{i,b} - \delta_{i,b}F_{a,j}
 \quad\text{for }1\le i,\ j,\ a,\ b\le p+q.
\]
Fix $s$, $t\in\mathbb C$ and let 
$\tau_{s,t}$ be the one dimensional representation of
$\mathfrak k_{\mathbb C}$ with $\tau_{s,t}(E_{\mu,\nu}) 
=\tau_{s,t}(E_{\bar i,\bar j}) = 0$ if
$\mu\ne\nu$ and $i\ne j$ and $\tau_{s,t}(E_{\nu,\nu})=s$ 
and $\tau_{s,t}(E_{\bar i,\bar i}) = t$.
Note that $\chi_\ell(X)=\tau_{s,t}(X)$ with $\ell=s-t$ 
for $X\in\mathfrak{k}_\mathbb{C}$ with $\trace X=0$. 
Put $$\tilde F_{u,v} = \sum_{w=1}^{p+q}E_{u,w}F_{w,v}.$$
Considering in modulo $\mathfrak n_{\mathbb C}U(\mathfrak g) + \sum_{X\in\mathfrak k_{\mathbb C}}
U(\mathfrak g)(X - \tau_{s,t}(X))$, we have
\begin{align*}
\allowdisplaybreaks
\tilde  F_{i,a} &= \sum_\nu E_{i,\nu}F_{\nu,a} + \sum_{i>j} E_{i,\bar j}F_{\bar j,a}
            + E_{i, \bar i}F_{\bar i,a} + \sum_{i<j} E_{i,\bar j}F_{\bar j,a}\\
          &\equiv \sum_{\nu}(F_{i,a}-\delta_{ia}F_{\nu,\nu})
            + sF_{i,a}
            + \sum_{i>j} E_{i,j}F_{\bar j,a}
            + \frac12(E_i + E_{i,i} - E_{\bar i,\bar i})F_{\bar i,a}\\
          &\quad
            -\sum_{i<j}E_{\bar i,\bar j}F_{\bar j,a}\\
          &\equiv (p+s)F_{i,a} - \delta_{ia}\sum_\nu F_{\nu,\nu}
                      - \delta_{ia}\sum_{i>j}F_{\bar j,j}
            + \frac{E_i + s-t-1}2F_{\bar i, a}
            - \frac12\delta_{ia}F_{\bar i,i}\\
          &\quad
            + \frac12\delta_{\bar ia}F_{\bar i,\bar i}
            - \sum_{i<j}(F_{\bar i,a} - \delta_{\bar ia}F_{\bar j,\bar j}),\\
 \tilde F_{k,a} &= \sum_\nu E_{k,\nu}F_{\nu,a} + \sum_{j} E_{k,\bar j}F_{\bar j,a}\\
          &\equiv\sum_\nu (F_{k,a} - \delta_{ka}F_{\nu,\nu})
           + sF_{k,a}
           + \sum_j E_{k,j}F_{\bar j,a}\\
          &= (p+s)F_{k,a} - \delta_{ka}\sum_\nu F_{\nu,\nu}
            - \sum_j \delta_{ka}F_{\bar j,j},\\
\tilde  F_{\bar i,a} 
          &= \sum_{i>j} E_{\bar i,j}F_{j,a} + E_{\bar i,i}F_{i,a}
            + \sum_{i<j} E_{\bar i,j}F_{j,a}
            + \sum_k E_{\bar i,k}F_{k,a}
            + \sum_j E_{\bar i, \bar j}F_{\bar j,a}\\
          &\equiv \sum_{i>j} E_{\bar i,\bar j}F_{j,a}
            + \frac12(E_i - E_{i,i} + E_{\bar i,\bar i})F_{i,a}
            - \sum_{i<j}E_{i,j}F_{j,a}
            - \sum_k E_{i,k}F_{k,a}\\
          &\quad
            + \sum_j(F_{\bar i,a} - \delta_{\bar i,a}F_{\bar j,\bar j})
            + tF_{\bar i,a}\\
         &\equiv 
            tF_{\bar i,a} -\delta_{\bar i,a}\sum_{i > j}F_{j,\bar j}
            + \frac{E_i - s + t - 1}2F_{i,a} - \frac12\delta_{\bar i,a}F_{i,\bar i}
            + \frac12\delta_{i,a}F_{i,i}\\
         &\quad
            - \sum_{i<j}(F_{i,a}  - \delta_{i,a}F_{j,j})
            - \sum_k(F_{i,a}- \delta_{i,a}F_{k,k}) 
            + \sum_j(F_{\bar i,a}
            - \delta_{\bar i,a}F_{\bar j,\bar j}).
\end{align*}
Suppose
\[
 F_{a, b} = 0\quad\text{if}\quad |a-b|\ne 0,\ p. 
\]
Then we have
\[
 \tilde F_{a, b} = 0\quad\text{if}\quad |a-b|\ne 0,\ p
\]
and
\begin{align*}
\allowdisplaybreaks
 \tilde F_{i,i} &= (p+s)F_{i,i} - \sum_{\nu=1}^pF_{\nu,\nu}
                   - \sum_{j<i}F_{\bar j,j} 
                   + \left(\frac{E_i + s - t}2 - q + i- 1\right)F_{\bar i,i}\\
                  &=sF_{i,i}
                   + \left(\frac{E_i + s - t}2 - q\right)F_{\bar i,i}
                   -  \sum_{\nu=1}^p(F_{\nu,\nu} - F_{i,i})
                   - \sum_{j=1}^{i-1}(F_{\bar j,j} - F_{\bar i.i}),\\
 \tilde F_{i,\bar i} &= (p+s)F_{i,\bar i}
                    + \left(\frac{E_i + s-t}2 - q + i\right)F_{\bar i,\bar i}
                    + \sum_{i<j}F_{\bar j, \bar j}\\
                    &= (p+s)F_{i,\bar i}
                    + \frac{E_i + s-t}2F_{\bar i,\bar i}
                    + \sum_{j=i+1}^q(F_{\bar j, \bar j} - F_{\bar i, \bar i}),\\
 \tilde F_{k,k} &= (p+s)F_{k,k} - \sum_{\nu =1}^pF_{\nu,\nu} 
                   - \sum_{j=1}^qF_{\bar j,j}\\
                &= sF_{k,k} - \sum_{j=1}^qF_{\bar j,j}
                    - \sum_{\nu =1}^p(F_{\nu,\nu} - F_{k,k}),\\
 \tilde F_{\bar i, i} &= \left(\frac{E_i-s+t}2 - p + i\right)F_{i,i} 
                   + \sum_{\nu=i+1}^pF_{\nu,\nu}
                   + (q+t)F_{\bar i,i}\\
                 &= (q+t)F_{\bar i,i}
                   + \frac{E_i-s+t}2F_{i,i}
                   + \sum_{\nu=i+1}^p( F_{\nu,\nu} - F_{i,i}),\\
 \tilde F_{\bar i,\bar i} &= -\sum_{j<i}F_{j,\bar j}
                   + \left(\frac{E_i - s + t}2 - p  + i - 1\right)F_{i,\bar i}
                   - \sum_{j\ne i}F_{\bar j,\bar j}
                   + (q+t-1)F_{\bar i,\bar i},\\
                &= tF_{\bar i,\bar i}
                   +  \left(\frac{E_i - s + t}2 - p\right)F_{i,\bar i}
                   - \sum_{j=1}^q(F_{\bar j, \bar j} - F_{\bar i,\bar i})
                   - \sum_{j=1}^{i-1}(F_{j,\bar j} - F_{i,\bar i}).
\end{align*}

Put
\[
 F_{i,\bar i}^1 = \frac{E_i + s - t}2,\quad
 F_{\bar i,i}^1 = \frac{E_i - s + t}2,\quad
 F_{i,i}^1 = F_{k,k}^1 = s + \lambda_1\quad
 \text{and}\quad
 F_{\bar i,\bar i}^1 = t + \lambda_1.
\]
Suppose $(u,v)$ are in $\{(i,i),\ (i, \bar i),\ (\bar i,i),\ (\bar i, \bar i),\ (k,k)\}$
and $F_{u,v}^{m-1}$ are defined.
By putting $F_{u,v} = F_{u,v}^{m-1}$, define $\tilde F_{u,v}^{m-1} = \tilde F_{u,v}$ by the above equations and moreover
\[
  F_{u,v}^m = \tilde F_{u,v}^{m-1} + \lambda_m F_{u,v}^{m-1}
  \in U(\mathfrak a_{\mathfrak p}).
\]
Thus we inductively define $F_{u,v}^m$.
Note that
\begin{align*}
  \bigl(F^m_{a,b}\bigr)
   _{\substack{1\le a\le p+q\\1\le b\le p+q}}
  &\equiv
  \prod_{j=1}^m\Bigl(
  \bigl(E_{a,b}+\lambda_j\delta_{a,b}\bigr)
   _{\substack{1\le a\le p+q\\1\le b\le p+q}}\Bigr)
  \\
  F_{i, \bar i}^{m} &= (\lambda_m + p + s)F_{i, \bar i}^{m-1}
       + \frac{E_i + s - t}2F_{\bar i,\bar i}^{m-1}
       + \sum_{j=i+1}^q
         (F_{\bar j, \bar j}^{m-1}-F_{\bar i, \bar i}^{m-1}),\\
 F_{\bar i, \bar i}^m &= 
        (t + \lambda_m)F_{\bar i, \bar i}^{m-1}
         - \sum_{j=1}^q
        (F_{\bar j, \bar j}^{m-1} - F_{\bar i, \bar i}^{m-1})\\
       &\quad
       + \left(\frac{E_i - s + t}2 - p\right)F_{i,\bar i}^{m-1}
       - \sum_{j=1}^{i-1}(F_{j, \bar j}^{m-1} - F_{i,\bar i}^{m-1}).\\
\end{align*}
Putting 
\[
  F_{\pm i}^m  = F_{\bar i,\bar i}^m \pm F_{i,\bar i}^m,
\]
we have
\begin{align*}
  2F_{\pm i}^m &= (\lambda_m + p + s)(F_{\pm i}^{m-1} - F_{\mp i}^{m-1})
       \pm\frac{E_i + s - t}2(F_{\pm i}^{m-1} + F_{\mp i}^{m-1})\\
       &\quad
       \pm\sum_{j=i+1}^q(F_{\pm j}^{m-1} - F_{\pm i}^{m-1})
       \pm\sum_{j=i+1}^q(F_{\mp j}^{m-1} - F_{\mp i}^{m-1})\\
       &\quad+(t+\lambda_m)(F_{\pm i}^{m-1} + F_{\mp i}^{m-1})
       -\sum_{j=1}^q(F_{\pm j}^{m-1} - F_{\mp i}^{m-1})
       -\sum_{j=1}^q(F_{\pm j}^{m-1} - F_{\mp i}^{m-1})\\
       &\quad\pm\left(\frac{E_i - s + t}2 - p\right)
        (F_{\pm i}^{m-1} - F_{\mp i}^{m-1})\\
       &\quad\mp\sum_{j=1}^{i-1}(F_{\pm j}^{m-1} - F_{\pm i}^{m-1})
       \pm\sum_{j=1}^{i-1}(F_{\mp j}^{m-1} - F_{\mp i}^{m-1})
\end{align*}
and
\begin{align*}
  F_i^m &= \left(\lambda_m + \frac{E_i + s + t}2\right)F_i^{m-1}
           - \sum_{j=1}^{i-1}(F_j^{m-1} - F_i^{m-1}),\\
 F_{-i}^m &= \left(\lambda_m + p - \frac{E_i + s + t}2\right)F_{-i}^{m-1} 
           - (p+s-t)F_i^{m-1} - \sum_{j=i+1}^q(F_{-j}^{m-1} - F_{-i}^{m-1}).
\end{align*}
For $0=n_0<n_1<\cdots<n_L=q$ and $(\mu_1,\ldots,\mu_L)\in\mathbb C^L$
put
\[
  E_i = 2\mu_\ell\quad\text{if there exists $\ell$ with}\quad\
  n_{\ell-1}<i\le n_\ell
\]
and
\begin{align}
  \lambda_k &=
   \begin{cases}
      -\mu_k -\dfrac{s + t}2 - n_{k-1}
          &\quad\text{if}\quad k\le L,\\
       \mu_{2L + 1 - k} -\dfrac{s+t}2 - p + n_{2L + 1 - k}
          &\quad\text{if}\quad L < k \le 2L.\
  \end{cases}\\
 f(x) &= \prod_{j=1}^{2L}\bigl(x+\lambda_k\bigr)
   =\prod_{k=1}^L\bigl(x-\mu_k -\tfrac{s + t}2 - n_{k-1}\bigr)
    \bigl(x+\mu_k -\tfrac{s+t}2 - p + n_k\bigr).
\end{align}
Then for $i>0$ inductively we can prove
\[
  F_i^m = 0\quad\text{if}\quad m\ge L\quad\text{or}\quad i \le n_m
\]
and moreover by the induction for $i=q$, $q-1,\ldots, 1$ 
\[
 F_{i,\bar i}^m = 0
 \quad\text{if}\quad m>L\quad\text{and}\quad i > n_{2L - m}.
\]
In particular we have $F_{i\bar i}^{2L} = 
F_{\bar i,\bar i}^{2L} = 0$ for $i=1,\ldots,q$ and hence
$F_{a,b}^{2L}=0$ for $a=1,\dots,p+q$ and $b=p+1,\dots,p+q$,

Note that when $p=q$, the same argument as above proves 
$F_{a,b}^{2L}=0$ also for $a=1,\dots,p+q$ and $b=1,\dots,p$.
\begin{lem}
Suppose $M = \Big(M_{ij}\Big)\substack{1\le i\le p+q\\1\le j\le p+q}
         \in M\bigl(p+q, U(\mathfrak g)\bigr)$
satisfies
\[
  [E_{ij}, M_{k\ell}] = \delta_{jk}M_{i\ell} - \delta_{\ell i}M_{k\ell}.
\]
Put $\tilde M = M\Big(E_{ij} + \lambda \delta_{ij}\Big)
    _{\substack{1\le i\le p+q\\1\le j\le p+q}}$
and $\tilde M' = \Big(E_{ij} + \lambda \delta_{ij}\Big)M$.
Then
\begin{align*}
   \tilde M_{a\nu}&\equiv 
                  M_{a \nu}\bigl(\lambda + s + q\bigr)
                  \mod \sum_{X\in\mathfrak k}U(\mathfrak g)(X-\chi_{s,t}(X))
                  + \sum_{\substack{1\le b \le p+q\\
                    p<c\le p+q}}
                   U(\mathfrak g)M_{bc}\\
    &\qquad\qquad\qquad\ \,\qquad\qquad\qquad\qquad
     \text{for}\quad 1\le a\le p+q,\ 1\le \nu\le p,\\
   \tilde M_{a\nu}'&\equiv 0
   \mod \sum_{\substack{1\le b \le p+q\\
                    p<c\le p+q}}
                   U(\mathfrak g)M_{bc}
   \qquad\text{for}\quad 1\le a\le p+q,\ p < \nu\le p+q.
\end{align*}
\end{lem}
\begin{pf} If $1\le\nu\le p$, then
\begin{align*}
\tilde M_{a\nu}^{m+1} &= \sum_{b=1}^{p+q}M_{a b}^m
                   (E_{b\nu}+\lambda\delta_{b\nu})\\
                 &\equiv
                     M_{a \nu}^m(\lambda + s)
                    + \sum_{b=p+1}^{p+q}M_{ab}^mE_{b\nu}
                  \mod \sum_{X\in\mathfrak k}U(\mathfrak g)(X-\chi_{s,t}(X))\\
                 &\equiv
                    M_{a \nu}^m(\lambda + s + q)
                  + \sum_{X\in\mathfrak k}U(\mathfrak g)\bigl(X-\chi_{s,t}(X)\bigr)
                  + \sum_{\substack{1\le b \le p+q\\
                    p<c\le p+q}}
                   U(\mathfrak g)M_{bc}^m.
\end{align*}
The latter relation is clear.
\qed
\end{pf}

Thus we have the following theorem.
\begin{thm}\label{thm:upq}
Put $\mathbb E=(E_{i,j})_{\substack{1\le i\le p+q\\1\le j\le p+q }}
\in M\bigl(p+q,\mathfrak g\bigr)$ and define 
\begin{equation}
  \tilde f(x) = (x-s-q)\prod_{k=1}^L
    \bigl(x-\mu_k -\tfrac{s + t}2 - n_{k-1}\bigr)
    \bigl(x+\mu_k -\tfrac{s+t}2 - p + n_k\bigr)
\end{equation}
and put
\begin{align}
  I^0_{\Xi}(\mu,s,t)
    &:=\sum_{1\le i\le p+q,\ p< j\le p+q}U(\mathfrak g)f(\mathbb E)_{i,j}
      =U(\mathfrak g)V_f^{0,1}+U(\mathfrak g)V_f^{1,1},\\
  I_{\Xi}(\mu,s,t)
    &:=\sum_{1\le i\le p+q,\ 1\le j\le p+q}U(\mathfrak g)f(\mathbb E)_{i,j}
      =U(\mathfrak g)V_f,\\
  \tilde I_{\Xi}(\mu,s,t)
    &:=\sum_{1\le i\le p+q,\ 1\le j\le p+q}U(\mathfrak g)\tilde f(\mathbb E)_{i,j}
      =U(\mathfrak g)V_{\tilde f}.
\end{align}
Then $I^0_{\Xi}(\mu,s,t)$ is a left ideal of 
$U(\mathfrak g)$ satisfying
\begin{equation}\label{eq:Upq}
 \begin{split}
 &D\equiv 0 \mod \mathfrak n_{\mathbb C} U(\mathfrak g)
  +\sum_{X\in\mathfrak k}U(\mathfrak g)\bigl(X-\chi_{s,t}(X)\bigr)\\
 &\qquad + 
   \sum_{i=1}^L\sum_{\nu=n_0+\cdots+n_{i-1}+1}^{n_0+\cdots+n_i}
   U(\mathfrak g)(E_\nu-2\mu_i) \quad(\forall D\in I^0_{\Xi}(\mu,s,t))
 \end{split}
\end{equation}
and $\tilde I_{\Xi}(\mu,s,t)$ is a two-side ideal of $U(\mathfrak g)$ satisfying
\begin{equation}
  \tilde I_{\Xi}(\mu,s,t)\subset I^0_{\Xi}(\mu,s,t) + 
  \sum_{X\in\mathfrak k}U(\mathfrak g)\bigl(X-\chi_{s,t}(X)\bigr).
\end{equation}
If $p=q$, the left ideal $I^0_{\Xi}(\mu,s,t)$ in the claim \eqref{eq:Upq}
may be replaced by the two sided ideal $I_{\Xi}(\mu,s,t)$.
\end{thm}
The polynomial $\tilde f(x)$ or $f(x)$ equals the minimal 
polynomial $q_{\Theta}(\mathfrak{gl}_{p+q};x,t)$ given in the last 
section when $p>q$ or $p=q$, respectively. 
Hence this theorem and the argument in the preceding section 
give the following corollary.
\begin{cor}
\label{cor:upq}
Suppose the infinitesimal character of 
$\mathcal B(G/P_\Xi;\mathcal L^\ell_{\Xi,\mu})$ is regular and
$c(\mu+\rho(\Xi),\ell)\ne 0$.
Then the image of $\mathcal P_{\Xi,\mu}^\ell$ is identified with
the subspace of $\mathcal A(G)$ killed by $I^0_{\Xi}(\mu,s,t)$
(resp.~$I_{\Xi}(\mu,s,t)$),
$\sum_{X\in\mathfrak k}U(\mathfrak g)(X-\chi_{s,t}(X))$
and $\Delta_i-c_i$ for $i=2,\dots,L-1$ 
with $\Delta_i=\trace \mathbb E^{i}$ when $p>q$ (resp.~$p=q$).
Here the complex parameters $\mu$, $s$, $t$ of $I^0_{\Xi}(\mu,s,t)$ 
or $I_{\Xi}(\mu,s,t)$ and $c_i\in\mathbb C$ are determined 
according to the parameter 
$\mu$ and $\ell$ of $\mathcal P_{\Xi,\mu}^\ell$,
\end{cor}
\begin{exmp}[Shilov boundary]\label{eq:UpqShilov}
Consider the case when $L=1$.
We will write
\[\mathbb E = \begin{pmatrix}K_1 & P \\ Q & K_2\end{pmatrix}\in 
M\bigl(p+q,U(\mathfrak g)\bigr)\]
for simplicity.  
Here $K_1=(E_{i,j})_{1\le i,\,j\le p}$ etc.
Then
\begin{align*}
 \mathbb E&\equiv\begin{pmatrix}s & P\\ Q & t\end{pmatrix},\\
 K_1P & \equiv (p+s)P
  \mod \sum_{X\in\mathfrak k}U(\mathfrak g)
  \bigl(X-\chi_{s,t}(X)\bigr),\\
 K_2Q & \equiv (q+t)Q   
  \mod \sum_{X\in\mathfrak k}U(\mathfrak g)
  \bigl(X-\chi_{s,t}(X)\bigr),\\
 \mathbb E^2&\equiv
   \begin{pmatrix}K_1 & P\\ Q & K_2\end{pmatrix}
   \begin{pmatrix}s & P\\ Q & t\end{pmatrix}
   =\begin{pmatrix}PQ+sK_1& (K_1+t)P \\
 (K_2+s)Q & QP+tK_2\end{pmatrix}\\
  &\equiv\begin{pmatrix}
  PQ + s^2 & (p+s+t)P \\ (q+s+t)Q & QP+t^2
 \end{pmatrix},\\
 \bigl(\mathbb E&-\lambda-\tfrac{s+t}2\bigr)
 \bigl(\mathbb E+\lambda-p-\tfrac{s+t}2\bigr)\\
 &\equiv
 \mathbb E^2 - \bigl(p+s+t\bigr)\mathbb E
 -\bigl(\lambda+\tfrac{s+t}2\bigr)
  \bigl(\lambda - p-\tfrac{s+t}2\bigr)\\
 &\equiv
 \begin{pmatrix}
  PQ+s^2-(p+\tfrac{s+t}2)s & 0\\
  (q-p) Q & QP + t^2 - t(p+\tfrac{s+t}2)
 \end{pmatrix}
  -\bigl(\lambda+\tfrac{s+t}2\bigr)\bigl(\lambda - p-\tfrac{s+t}2\bigr)\\
 &\equiv 
 \begin{pmatrix}
  PQ - s(p+t) & 0 \\ (q-p)Q & QP - t(p+s)
 \end{pmatrix}
 -\bigl(\lambda+\tfrac{s+t}2\bigr)\bigl(\lambda-p-\tfrac{s+t}2\bigr)\\
 &=\begin{pmatrix}
 PQ - (s-t)p &  0\\
 (q-p)Q      &  QP
 \end{pmatrix} -
 \bigl(\lambda+\tfrac{s-t}2\bigr)\bigl(\lambda - p -\tfrac{s-t}2\bigr).
\end{align*}
Then the system of the equations of the second order characterizing the image 
of the corresponding
Poisson transform equals
\begin{equation}
  (QP)_{i,j}u =
  \delta_{i,j}
  \bigl(\lambda+\tfrac{s-t}2\bigr)\bigl(\lambda - p -\tfrac{s-t}2\bigr)
  u\quad(1\le i,j\le q).
\end{equation}
Note that the element $\trace QP$ of $U(\mathfrak g)$ defines a 
$G$-invariant differential operator on the homogeneous 
line bundle ${\mathbf E}_{\ell}$ over $G/K$ with $\ell=s-t$, 
which is a constant multiple of the Laplace-Beltrami operator 
on $\mathbf{E}_\ell$.
\end{exmp}
\begin{rem}\label{rem:upq}
The second order operators $(QP)_{i,j}$ in
Example~\ref{eq:UpqShilov} are nothing but the so called
Hua operators for $G=U(p,q)$.
In the case of the trivial line bundle over $G/K$,
that is the case of $s=t=0$,  the fact that
they characterize the image of the Poisson transform
on the Shilov boundary was proved in \cite{JK} for
$p=q$ and $\lambda=p$, \cite{Sn2} for $p=q$ and generic $\lambda$,
\cite{BV} for $p>q$ and $\lambda=p$, and
\cite{KZ} for $p>q$ and generic $\lambda$.
Our result gives a further generalization to line bundles over 
$G/K$. Moreover the differential operators of order $2L$ 
corresponding to  $G/P_\Xi$ in
Corollary~\ref{cor:upq} 
can be considered to be a generalization of the second
order Hua operators corresponding to the Shilov boundary. 
\end{rem}

\subsection{$Sp(n,\mathbb R)$}
We calculate the system of differential equations characterizing the
image of the Poisson transform $\mathcal P_{\Xi,\mu}^\ell$ attached
to the Shilov boundary of the symmetric space $Sp(n,\mathbb R)/U(n)$
as in the case of the symmetric space $U(p,q)/U(p)\times U(q)$.
Putting
\[\mathbb F = \begin{pmatrix}
               K & P \\ Q & -{}^tK
             \end{pmatrix}
\quad\text{ with }
 \begin{cases}
   2K_{ij} = E_{ij} - E_{j+n,i+n},\\
   2P_{ij} = E_{i,j+n} + E_{j,i+n},\\
   2Q_{ij} = E_{i+n,j} + E_{j+n,i},
 \end{cases}
\]
we have $\sum_{1\le i,\,j\le 2n}\mathbb C F_{ij}\simeq\mathfrak {sp}_{n}$.
\begin{align*}
 [E_{ij}-E_{j+n,i+n},E_{k,\ell+n}+E_{\ell,k+n}]&
 =  \delta_{jk}E_{i,\ell+n}+\delta_{j\ell}E_{i,k+n}
  + \delta_{j\ell}E_{k,i+n}+\delta_{jk}E_{\ell, i+n} 
 \\
 [K_{ij}, P_{k\ell}] &= \tfrac12\delta_{jk}P_{i\ell}+\tfrac12\delta_{j\ell}P_{ik},\\
 [K_{ij},Q_{k\ell}] &= -\tfrac12\delta_{ik}Q_{j\ell}-\tfrac12\delta_{i\ell}Q_{jk},\\
 \sum_\nu K_{i\nu}P_{\nu j} &- \sum_\nu P_{\nu j}K_{i\nu}
  = \tfrac n2 P_{ij}+\tfrac12 P_{ij}=\tfrac{n+1}2P_{ij},\\
 \sum_\nu -K_{\nu i}Q_{\nu j} & + \sum_\nu Q_{\nu j}K_{\nu i}
  = \tfrac n2 PQ_{ij}+\tfrac12 Q_{ij}=\tfrac{n+1}2Q_{ij},\\
\begin{pmatrix}K & P\\ Q & -{}^t\!K\end{pmatrix}
\begin{pmatrix}\ell & P\\ Q & -\ell\end{pmatrix}
 &=\begin{pmatrix}\ell K+PQ & KP-\ell P\\
 \ell Q-{}^t\! KQ & QP+\ell\,{}^t\!K\end{pmatrix}\\
 &\equiv \begin{pmatrix}PQ +\ell^2 & \tfrac{n+1}2P\\\tfrac{n+1}2Q& QP+\ell^2
 \end{pmatrix}\\
 \mathbb E\bigl(\mathbb E-\tfrac{n+1}2\bigr)&\equiv
 \begin{pmatrix}PQ + \ell(\ell-\tfrac{n+1}2)& 0\\ 0& QP
  + \ell(\ell+\tfrac{n+1}2)\end{pmatrix}\\
 \bigl(\mathbb E-\lambda\bigr)\bigl(\mathbb E+\lambda-\tfrac{n+1}2\bigr)&\equiv
 \begin{pmatrix}PQ - (n+1)\ell & 0
       \\ 0& QP
 \end{pmatrix}\\
 &\quad - \bigl(\lambda+\ell\bigr)\bigl(\lambda-\ell-\tfrac{n+1}2\bigr).
\end{align*}
Hence the system of the differential equations are
\begin{equation}
\label{eqn:sphua}
 \begin{cases}
  (PQ)_{i,j}u = \delta_{i,j}
     \bigl(\lambda-\ell\bigr)\bigl(\lambda+\ell-\tfrac{n+1}2\bigr)u
       &(1\le i,\,j\le n),\\
  (QP)_{i,j}u = \delta_{i,j}\bigl(\lambda+\ell\bigr)
     \bigl(\lambda-\ell-\tfrac{n+1}2\bigr)
     u&(1\le i,\,j\le n).
 \end{cases}
\end{equation}
\begin{rem}
The second order operators $(PQ)_{i,j}$ and $(QP)_{i,j}$
 are nothing but the so called
Hua operators for $G=Sp(n,\mathbb{R})$.
The fact that the  equations (\ref{eqn:sphua}) characterize the image
of the Poisson transform on the Shilov boundary
was proved by the second author~\cite{Sn3} for
generic $\ell$ and $\lambda$.
In the case of the trivial line bundle over $G/K$,
that is the case of $\ell=0$,  it  was proved in \cite{KM} for
$n=2$ and $\lambda=\frac32$, \cite{J1} for $\lambda=\frac{n+1}{2}$,
and \cite{Se} for generic $\lambda$.
\end{rem}
\subsection{$GL(n,\mathbb R)$}
Lastly we give a lemma which helps the calculation of the system of
differential operators on the symmetric space $GL(n,\mathbb R)_+/SO(n,\mathbb R)$.
Here $GL(n,\mathbb R)_+=\{g\in GL(n,\mathbb R)\,;\,\det g>0\}$.
\begin{lem}[$GL(n,\mathbb R)$] Put
\begin{align*}
 \mathbb E&=K+P=\Bigl(K_{ij}+P_{ij}\Bigr)\quad\text{ with }
 \begin{cases}
 K_{ij} = \tfrac12(E_{ij}-E_{ji}),\\
 P_{ij} = \tfrac12(E_{ij} + E_{ji}),
 \end{cases}\\
 \mathfrak g_{\mathbb C}&=\sum_{i,\,j=1}^n\mathbb CE_{ij} \simeq \mathfrak{gl}_n,
 \text{ \ and \ }
 \mathfrak k_{\mathbb C}=\sum_{i,\,j=1}^n \mathbb CK_{ij} \simeq \mathfrak{o}_n.
\end{align*}
Then for $m=0,1,2\ldots$
\begin{align}
  KP^m &= \tfrac n2P^m - \tfrac12 \trace(P^m)
         +\sum_{\nu=1}^n (P^m)_{\nu j}K_{i\nu}\label{eq:gln1}\\
       &\equiv \tfrac n2P^m - \tfrac12 \trace(P^m)
         \mod U(\mathfrak g)\mathfrak k,\notag\\
 (\mathbb E-\tfrac n2)P^m&\equiv P^{m+1}-\tfrac12\trace(P^m),\label{eq:gln2}\\
P^m &\equiv(\mathbb E-\tfrac n2)^{m-1}\mathbb E+
      \tfrac12\sum_{k=2}^m(\mathbb E-\tfrac n2)^{m-k}\trace(P^{k-1}),\label{eq:gln3}\\
\trace(P^m)
    &\equiv \trace\bigl((\mathbb E-\tfrac{n-1}2)^{m-1}\mathbb E\bigr).
     \label{eq:gln4}
\end{align}
\end{lem}
\begin{proof}
Since
\begin{align*}
[E_{ij},E_{kl}] &=
 \delta_{jk}E_{il} - \delta_{l i}E_{k j},\\
[E_{ij},E_{kl}+E_{kl}] &=
 \delta_{jk}E_{il} - \delta_{l i}E_{k j}+
 \delta_{jl}E_{ik} - \delta_{k i}E_{l j},\\
[E_{ij}-E_{ji},E_{kl}+E_{kl}] &=
 \delta_{jk}E_{il} - \delta_{l i}E_{k j}+
 \delta_{jl}E_{ik} - \delta_{k i}E_{l j}\\
 &\quad{}
 -\delta_{ik}E_{jl} + \delta_{l j}E_{k i}
 - \delta_{il}E_{jk} + \delta_{k j}E_{l i}\\
 &=2(\delta_{jk}P_{il}+\delta_{jl}P_{ik}-
     \delta_{il }P_{jk}-\delta_{ik}P_{jl})\\
[K_{ij},P_{kl}] &=
 \tfrac12(\delta_{jk}P_{il}+\delta_{jl}P_{ik}
 -\delta_{ik}P_{jl}-\delta_{il}P_{jk}),\\
\sum_{\mu,\nu}(P^p)_{\mu\nu}K_{i\mu}(P^q)_{\nu j}
 &-\sum_{\mu,k}(P^{p+1})_{\mu k}K_{i\mu}(P^{q-1})_{kj}\\
 &=\sum_{\mu,\nu,k}(P^p)_{\mu\nu}[K_{i\mu},P_{\nu k}](P^{q-1})_{k j}\\
 &= 
  \tfrac12\sum_{\mu,\nu,k}(P^p)_{\mu\nu}(\delta_{\mu\nu}P_{ik}+\delta_{\mu k}
  P_{i\nu}-\delta_{i\nu}P_{\mu k}-\delta_{ik}P_{\mu\nu})(P^{q-1})_{kj}\\
 &= 
  \tfrac12(\trace P^p)(P^q)_{ij}
  -\tfrac12(\trace P^{p+1})(P^{q-1})_{ij},
\end{align*}
we have \eqref{eq:gln1} by the sum of these equations
for $p=m-q=0,\ldots,m$.
Moreover \eqref{eq:gln2} follows from \eqref{eq:gln1}.
Then \eqref{eq:gln2} proves \eqref{eq:gln3} by the induction on $m$
and \eqref{eq:gln4} corresponds to the trace of the matrices in
\eqref{eq:gln3}.
\end{proof}
\begin{rem}
Our study of  characterizing the images of Poisson transform for general
boundaries of a symmetric space by two sided ideals 
is originated in \cite{O4} for the boundaries of $GL(n,\mathbb R)_+/SO(n)$, 
where generators of the ideals that are different from minimal 
polynomials are constructed. 
The ideal spanned by the components of $\gamma\bigl(f(\mathbb E)\bigr)$ 
for any polynomial $f(x)$ is calculated by \cite[Theorem~4.19]{O6}
for the symmetric space $GL(n,\mathbb C)/U(n)$.
\end{rem}

\end{document}